\documentclass[12pt]{amsart}

\usepackage{amscd}
\usepackage{amssymb}
\usepackage[all]{xy}

\swapnumbers

\theoremstyle{plain}

\numberwithin{equation}{section}
\newtheorem {Theorem}[equation]{Theorem}
\newtheorem {Lemma}[equation]            {Lemma}
\newtheorem {Corollary}[equation]       {Corollary}

\newtheorem{Proposition}[equation]       {Proposition}
\newtheorem*{Proposition*}              {Proposition}

\theoremstyle{definition}
\newtheorem{Definition}[equation]{Definition}
\newtheorem{Notation}[equation]{Notation}

\theoremstyle{remark}
\newtheorem{Remark}[equation]{Remark}
\newtheorem*{remark}{Remark}
\newtheorem{Example}[equation]{Example}

\newcommand     {\comment}[1]   {}
\newcommand{\mute}[2] {}
\newcommand     {\printname}[1] {}

\newcommand{\labell}[1] {\label{#1}\printname{#1}}

\def \dash {\textbf{--}}

\DeclareMathOperator{\image}{image}

\newcommand{\cut}{\text{cut}}
\newcommand{\ext}{\text{ext}}

\newcommand{\Z}{{\mathbb Z}}
\newcommand{\R}{{\mathbb R}}
\newcommand{\C}{{\mathbb C}}

\newcommand{\calT}{{\mathcal T}}

\newcommand{\calO}{{\mathcal O}}
\newcommand{\calJ}{{\mathcal J}}

\newcommand{\fh}{\mathfrak{h}}
\newcommand{\fg}{\mathfrak{g}}
\newcommand{\ft}{\mathfrak{t}}

\newcommand{\Rplus}{\R_{>0}}
\newcommand{\eps}{\varepsilon}

\newcommand{\ssminus}{\smallsetminus}
\newcommand{\Inv}{^{-1}}

\newcommand{\ol}{\overline}
\newcommand{\gammabar}{\ol{\gamma}}
\newcommand{\Psibar}{\ol{\Psi}}
\newcommand{\xbar}{\ol{x}}
\newcommand{\Abar}{\ol{A}}
\newcommand{\Lambdabar}{\ol{\Lambda}}
\newcommand{\sigmabar}{\ol{\sigma}}

\newcommand{\tJ}{\widetilde{J}}
\newcommand{\ts}{\tilde{s}}
\newcommand{\tlambda}{\tilde{\lambda}}
\newcommand{\tgamma}{\tilde{\gamma}}
\newcommand{\talpha}{{\widetilde{\alpha}}}
\newcommand{\tpsi}{\widetilde{\psi}}

\newcommand{\xhat}{\hat{x}}
\newcommand{\gammahat}{\hat{\gamma}}

\newcommand{\half}{{\frac{1}{2}}}
\newcommand{\del}{\partial}

\newcommand{\co}{\colon\thinspace}

\begin{document}

	\title{Convexity package for momentum maps on contact manifolds}

	\author[River Chiang]{River Chiang}
	\address{Department of Mathematics, National Cheng Kung University, Tainan 701, Taiwan}
	\email{riverch@mail.ncku.edu.tw}

	\author[Yael Karshon]{Yael Karshon}
	\address{Department of Mathematics, University of Toronto, Toronto, Ontario
	M5S 2E4, Canada}
	\email{karshon@math.toronto.edu}


\begin{abstract}
Let a torus $T$ act effectively on a compact connected cooriented 
contact manifold, and let $\Psi$ be the natural momentum map 
on the symplectization.
We prove that, if $\dim T > 2$, 
the union of the origin with the image of $\Psi$
is a convex polyhedral cone, the nonzero level sets of $\Psi$
are connected (while the zero level set can be disconnected),
and the momentum map is open as a map to its image.
This answers a question posed by Eugene Lerman,
who proved similar results when the zero level set
is empty.  We also analyze examples with $\dim T \leq 2$.
\end{abstract}




\maketitle

\setcounter{tocdepth}{1}
\tableofcontents

\section{Introduction}
\labell{sec:intro}

One of the fundamental theorems in equivariant symplectic geometry
is the convexity theorem for Hamiltonian torus actions.
This theorem is part of the following ``convexity package".
Let $\Phi \co M \to \R^k$ be a momentum map 
for a Hamiltonian torus action
on a compact connected symplectic manifold.
Then $\Phi$ has these properties:
\begin{itemize}
\item[(S1)]
The image $\Phi(M)$ is a convex polytope.
\item[(S2)]
The level sets of $\Phi$ are connected.
\item[(S3)]
The map $\Phi$ is open as a map to its image.
\end{itemize}
See~\cite{At,GS,Sj}.

Eugene Lerman~\cite{lerman} gave an analogous theorem
in equivariant contact geometry when the torus orbits
are transverse to the contact distribution,
and asked whether the transversality condition is necessary  
(cf.\ Remark~\ref{transverse}).
In this paper we answer Lerman's question
and give a ``convexity package" for momentum maps on contact manifolds.
More precisely, let $M$ be a compact connected cooriented contact manifold,
equipped with an effective action of a torus of dimension $k > 2$, 
and let
$$\Psi \co M \times \Rplus \to \R^k$$
be the momentum map on the symplectization (see below).
The \emph{momentum cone} is
$$ C(\Psi) := \{ 0 \} \cup \Psi(M \times \Rplus) .$$
Then $\Psi$ has these properties:
\begin{itemize}
\item[(C1)]
The momentum cone $C(\Psi)$ is a convex polyhedral cone.
\item[(C2)]
The nonzero level sets, $\Psi\Inv(\mu)$, for $\mu \neq 0$, are connected.
\item[(C3)]
The map $\Psi$ is open as a map to its image.
\end{itemize}
See Theorem~\ref{contact main}.

\bigskip

We now recall relevant definitions.

Let $M$ be a manifold of dimension $2n+1$.
A \emph{contact form} on $M$ is a one--form $\alpha$
such that $\alpha \wedge (d\alpha)^n$ never vanishes,
or, equivalently, such that $d\alpha$ is nondegenerate
on $\ker \alpha$.
If $\alpha$ is a contact one--form and $f$ is a positive function,
then $f\alpha$ is a contact one--form and $\ker \alpha = \ker f\alpha$.

A \emph{contact structure} $\xi$ on $M$ is a codimension one distribution
(subbundle of the tangent bundle $TM$) that can be
locally obtained as the kernel of a contact one--form.
If $\xi$ is cooriented, there exists a globally defined one--form $\alpha$
such that $\xi = \ker \alpha$ and $\alpha$ induces the coorientation
of $\xi$.  Such $\alpha$ is unique up to multiplication by a positive function.

The \emph{symplectization} of $(M,\alpha)$ is the symplectic manifold
$(M \times \Rplus , d(t\alpha) )$,
where $t$ is the coordinate on $\Rplus$ and where we use the same symbol
$\alpha$ to denote the contact one--form on $M$ and its pullback
to $M \times \Rplus$.
Nondegeneracy of $d(t\alpha)$ follows from the property
$\alpha \wedge (d\alpha)^n \neq 0$.

Consider the positive connected component of the annihilator of $\xi$
in the cotangent bundle $T^*M$:
$$ \xi^0_+ := \left\{ (x,\beta) \ | \
   x \in M, \ \beta \in T^*_xM, \ \beta(\xi|_x) = 0, \
   \beta \text{ induces the coorientation of $\xi|_x$ }
\right\} .$$
The map $M \times \Rplus \to \xi^0_+$
which sends $(x,t)$ to $t \alpha_x$
defines a trivialization of $\xi^0_+$
as a principal $\Rplus$ bundle.
Also, it pulls back the tautological one--form on $T^*M$
to the one--form $t\alpha$ on $M \times \Rplus$
and the standard symplectic form on $T^*M$
to the symplectic form $d(t\alpha)$ on $M \times \Rplus$.
Thus, to avoid choosing a contact one--form, one may define the
symplectization of $(M,\xi)$ to be the symplectic submanifold $\xi^0_+$
of $T^*M$.

Let a torus $T \cong (S^1)^k$ act on $M$ and preserve 
the cooriented distribution $\xi$.
Let $X_M$, for $X$ in the Lie algebra $\ft$ of $T$, 
denote the vector fields on $M$
that are induced from the action:
$X_M(x)=\frac{d}{dt}|_{t=0}(\exp tX) \cdot x$.
The action naturally lifts to a Hamiltonian action
on the cotangent bundle $T^*M$
which preserves the submanifold $\xi^0_+$.
The standard momentum map on $T^*M$
restricts to a momentum map on $\xi^0_+$
whose $X$ component, for $X \in \ft$,
is given by $(x,\beta) \mapsto \beta(X_M(x))$
for all $x\in M$ and $\beta\in \xi^0_+|_x \subset T^*_x M$.

Let $\alpha$ be a $T$--invariant contact one--form on $M$.
(Such $\alpha$ can be obtained by averaging; see Lemma~2.6
of Lerman's paper \cite{lerman:ctm}.)
The \emph{$\alpha$--momentum map} is the map
$\Psi_\alpha \co M \to \ft^*$
whose $X$ component $\Psi_\alpha^X \co M \to \R$,
for $X \in \ft$, is given by
\begin{equation} \labell{Psi alpha}
 \Psi_\alpha^X(x) = \alpha(X_M(x)) 
\end{equation}
for all $x \in M$.
It satisfies 
$d\Psi_\alpha^X = -\iota(X_M)d\alpha$ on $M$.
Using $\alpha$ to identify $\xi^0_+$ with $M \times \Rplus$,
the induced $T$ action on $M \times \Rplus$ is the given action
on the $M$ component and is trivial on the $\Rplus$ component,
and the momentum map becomes the map
$$ \Psi \co M \times \Rplus \to \ft^* \qquad , \qquad
    (x,t) \, \mapsto \,  t\Psi_\alpha(x) , $$
which we call the \emph{contact momentum map} corresponding to $\alpha$.
Thus, the momentum cone is
\[
\begin{split}
C(\Psi) & = \{ 0 \} \cup \Psi(M \times \Rplus) \\
 & = \R_{\geq 0} \cdot \Psi_\alpha(M).
\end{split}
\]

\begin{Remark} \labell{transverse}
The action is said to be \emph{transverse} if the orbits are transverse
to the contact distribution.  By the formula~\eqref{Psi alpha}
for the momentum map, an action is transverse if and only if 
its momentum map never takes the value zero.
A contact momentum map $\Psi \co M \times \Rplus \to \ft^*$
is never proper as a map to $\ft^*$: the preimage
of a closed ball centered at the origin is never compact.
But if $M$ is compact and the action is transverse,
then the momentum map is proper as a map to $\ft^* \ssminus \{ 0 \}$.
In the case of transverse torus actions,
parts (3) and (4) of Theorem~\ref{contact main} were proved 
by Eugene Lerman in~\cite{lerman}.
\end{Remark}

\begin{Remark}
The definition of the contact momentum map,
as a map on $M \times \Rplus$, 
depends on the choice of one--form $\alpha$.
The topological properties of the contact momentum map
are independent of the choice of $\alpha$,
because we can work directly on $\xi^0_+$.
\end{Remark}

\begin{Remark} \labell{symplectic proper}
If the image of the contact momentum map is contained in an open half--space,
then the ``convexity package" (C1), (C2), (C3)
is true without the dimension assumption on $T$, 
and it follows from the ``convexity package"
for symplectic manifolds with proper momentum maps:
\begin{quotation}
Let $\Phi \co Q \to \ft^*$ be a momentum map for a torus action
on a connected symplectic manifold.
Suppose that there exists a convex subset $\calT$ of $\ft^*$ 
that contains the image $\Phi(Q)$ and such that $\Phi \co Q \to \calT$ 
is proper. Then $\Phi$ is open as a map to its image,
the image $\Phi(Q)$ is a convex polyhedral subset of $\calT$, 
and the level sets of $\Phi$ are connected.
\end{quotation}
See \cite{CDM}, \cite{HNP}, \cite{prato}, \cite[Theorem~4.3]{LMTW},
and \cite[Theorem~30]{karshon-marshall}.
\end{Remark}

\begin{Remark} \labell{counterexamples}
Under our assumptions that guarantee (C1), (C2), and (C3)
(in particular, $\dim T > 2$),
the zero level set \emph{can} be disconnected;
see Examples~\ref{disconnected} and~\ref{disconnected2}.
Eugene Lerman noted, in \cite[Remark~1.4]{lerman}
and referring to his construction in~\cite{lerman:contact-cuts},
that if $\dim T \leq 2$,
then (C2) may fail, and if $\dim T = 2$, then (C1) may fail. 
We analyze Lerman's examples in more detail
in Example~\ref{ex:dim two},
where $\dim T = 2$ and the zero level set is empty,
and in Example~\ref{level},
where $\dim T = 1$ and the zero level set is nonempty.
If $\dim T=1$ and the zero level set is empty,
(C1), (C2), (C3) all hold, by Remark~\ref{symplectic proper}.
In fact, (C1) holds trivially if $\dim T=1$.
If $\dim T =2$ and the zero level set is nonempty,
we are currently unaware of examples where (C1), (C2), and (C3)
don't all hold. 
\end{Remark}

The ``convexity package" (C1), (C2), (C3)
says that the momentum map has certain global 
topological properties.  
To prove it, 
we show that the momentum map has certain local 
topological properties,
and we show that these local properties imply the required
global properties
by means of a point--set--topological ``local to global" argument.
We now give some details.

Let $X$ denote a compact connected Hausdorff space.

We define a map $\varphi \co X \to \R^n$ 
to be \emph{convex} if any two points in $X$
can be connected by a path whose composition 
with $\varphi$ is a weakly monotone parametrization
of a (possibly degenerate) segment.
See Section~\ref{sec:convex maps}.
Similarly, we define a map $\Psi \co X \to S^{n-1}$
to be \emph{spherically convex}
if any two points in $X$ whose images are not antipodal
can be connected by a path whose composition with $\Psi$
is a weakly monotone geodesic of length $< \pi$.
The image of such a map is spherically convex,
and the level sets of such a map are connected. 
See Remark~\ref{notions of convexity} 
and Definition~\ref{spherically convex map}.

In Sections~\ref{sec:spherical geometry}--\ref{sec:shorten}
we give a ``local to global" argument for maps from $X$ to $S^{n-1}$.
If the image is not contained in a great circle,
and if every point has a neighborhood on which the map
is spherically convex and open to its image,
then the map is spherically convex and open to its image.
This result, together with the analysis of the case
that the image \emph{is} contained in a great circle,
is given in Proposition~\ref{two possibilities}.

In Section~\ref{sec:nonvanishing}, we give a 
``local to global" argument for a map 
$\Psi \co X \times \Rplus \to \R^n$
that is obtained from a nonvanishing map
$\varphi \co X \to \R^n$ 
by $\Psi(x,\lambda) = \lambda \varphi(x)$.
Suppose that the image is not contained in a two dimensional subspace,
and suppose that every point in $X \times \Rplus$ has a neighborhood
on which $\Psi$ is convex and $\Psi/\| \Psi \|$ is open to its image.
Then every two points $y_0$ and $y_1$ in $X \times \Rplus$
for which
the segment $[\Psi(y_0),\Psi(y_1)]$ does not contain the origin
can be connected by a path whose composition
with $\Psi$ is a weakly monotone parametrization of this segment.
Additionally, $\Psi$ is open as a map to its image. 
See Proposition~\ref{stretch}.

In Section~\ref{sec:excision} we define $\Psi$ as before
but we allow $\varphi$ to sometimes vanish.
This is the case that applies to contact momentum maps 
for nontransverse actions.
The strategy is to remove from $X$ the $\varphi$--preimage
of an open ball around the origin
and to apply the results of Section~\ref{sec:nonvanishing} to 
the resulting ``excised space".
The precise statement is more technical than in the nonvanishing case,
and the proof is more involved.
We formulate local assumptions on $\varphi$ and $\Psi$
that guarantee the following global properties.
Every two points in $X \times \Rplus$ 
that do not both lie on the zero level set
can be connected by a path whose composition with $\Psi$
is a weakly monotone parametrization of a 
(possibly degenerate) segment; additionally,
the map $\Psi$ is open as a map to its image.
See Proposition~\ref{with zero'}.

In Section~\ref{sec:contact} we show that momentum maps 
on symplectizations have the local openness and convexity properties
that, by the results of Section~\ref{sec:excision},
imply the global properties that we had set to prove.
These local properties are consequences of the local normal form 
theorem, which describes the neighborhood of an orbit
in a symplectic manifold with a Hamiltonian action of a compact Lie group.
We apply the local normal form theorem in two different ways:
\begin{enumerate}
\item
The symplectization is a symplectic manifold
with a Hamiltonian torus action. We apply the
local normal form theorem to neighborhoods of its orbits.
\item
Take the contact manifold itself, with a momentum map
associated to some choice of invariant contact one--form. 
Consider an orbit that lies on the zero level set
of the momentum map.
By formula~\eqref{Psi alpha}, 
the tangent space
to this orbit is contained in the contact distribution;
therefore it is transverse to the Reeb directions.
We find an invariant codimension one submanifold
containing the orbit and transverse to the Reeb directions, 
on which the differential of the contact form is a nondegenerate two--form. 
We apply the local normal form to this submanifold.  
\end{enumerate}

To show the local properties of the momentum map,
it remains to examine the local models that appear in the
local normal form theorem.  We use the additional fact 
that the momentum map image of the orbit
belongs to the annihilator of the stabilizer of the orbit;
it holds because the momentum map comes from 
a one--form whose differential is the symplectic form.

By examining the local models,
we need to show that the restriction
of a linear projection to a standard simplex 
with a facet removed is open as a map to its image
and has the weak path lifting property. 
These properties may be intuitively obvious but their rigorous proofs
are not entirely trivial;
we give them in Section~\ref{sec:linear}.
\comment{
We also need to show that the normalized momentum map
on the local model
is open as a map to its image in the sphere;
we do this in Lemma~\ref{model sph open}.
}

We conclude Section~\ref{sec:contact} with 
a proof of the convexity package; see Theorem \ref{contact main}.
And, in Section~\ref{sec:examples}, we conclude the paper with examples.

\medskip

The type of local--to--global technique that we use was initiated 
by Condevaux--Dazord--Molino in \cite{CDM} 
and developed in \cite{HNP,HNP2,BOR1,BOR2,karshon-marshall}
by Hilgert--Neeb--Plank, Birtea--Ortega--Ratiu, and Bjorndahl--Karshon;
also see~\cite{knop}.
It can be viewed as a generalization of the Tietze--Nakajima
theorem~\cite{nakajima, tietze}.

Convexity results in contact geometry appeared in~\cite{bm2}
(toric case) and in~\cite{lerman} (transverse case).
Torus actions on contact manifolds (or on symplectic cones,
cf.\ Remark~\ref{symplectic cone})
were also studied in 
\cite{GS:homogeneous,
albert,kamishima-tsuboi,mt,boyer-galicki,loose,lerman:contact-cuts,
lerman-willett,lerman,willett,lerman:ctm,nozawa}.

The book \cite{guillemin-sjamaar} gives an overview 
of convexity results in symplectic geometry.
The book \cite{geiges-book} is a general reference
for (not necessarily equivariant) contact geometry.
Finally, 
our notion of convexity is unrelated to ``contact convexity"
as in, e.g., \cite[Chapter~5]{geiges-book}. 

\subsection*{Acknowledgements}
We are grateful to Eugene Lerman for teaching us
about momentum maps for contact manifolds.
Y.\ K.\ is grateful to Shlomo Sternberg for directing her 
to the paper by Condevaux--Dazord--Molino, many years ago.

This research was partially supported by an NSERC Discovery grant,
by the NSC grant 96-2115-M-006-010-MY2, and by the NCTS (South).

\section{Convex maps}
\labell{sec:convex maps}

\begin{Definition}
A path $\gamma \co [a,b] \to \R^n$ is \emph{weakly monotone straight}
if, for any $t_1,t_2,t_3 \in [a,b]$,
if $t_1 \leq t_2 \leq t_3$, then
$\gamma(t_2) \in [\gamma(t_1),\gamma(t_3)]$.
\end{Definition}

\begin{Definition} \labell{convex map}
A map $\varphi$ from a topological space $X$ to $\R^n$ or to a convex
subset of $\R^n$ is a \emph{convex map} if every two points $x_0$, $x_1$
in $X$ can be connected by a path $\gamma \co [0,1] \to X$
with $\gamma(0) = x_0$ and $\gamma(1) = x_1$,
such that the composition
$\varphi \circ \gamma \co [0,1] \to \R^n$ is weakly monotone straight.
\end{Definition}

Clearly, if $\varphi \co X \to \R^n$ is convex,
then $X$ is connected, the image of $\varphi$ is convex,
and the level sets of $\varphi$ are connected.

\begin{remark}
Our definition of a convex map, from a topological space to $\R^n$,
\emph{does not} agree with the notion
of a ``convex function'', from a convex subset of a vector space to $\R$,
being a function that satisfies
$f(tx+(1-t)y)\leq t f(x)+(1-t)f(y)$ for all $0 \leq t \leq 1$.
\end{remark}

We recall Theorem 15 of~\cite{karshon-marshall}:

\begin{Theorem} \labell{local-global}
Let $X$ be a connected Hausdorff topological space,
let $\calT$ be a convex subset of $\R^n$,
and let $\Psi \co X \to \calT$ be a continuous
proper map.  Suppose that for every point $x \in X$
there exists an open neighborhood $U$ of $x$
such that the map $\Psi|_U \co U \to \Psi(U)$
is convex and open.  Then the map $\Psi \co X \to \Psi(X)$
is convex and open.
\end{Theorem}

\begin{Example} \labell{example R2}
In each of the following examples, $X$ is a subset of $\R^2$,
and $\Psi \co X \to \R$ is the projection to the $x$--coordinate.
In each of these examples, exactly one of the assumptions
of Theorem~\ref{local-global} fails,
and the map $\Psi$ is not convex.

\begin{enumerate}

\item
Let $X$ be the union of the set where $x>0$ and $|y| \leq 1$
and the set where $x=0$ and $|y|=1$.
The projection to the $x$--axis is locally convex and locally open
to its image, but it is not proper.

\item
Let $X$ be the union of the negative $x$--axis,
the set where $x \geq 0$ and  $y = x$,
and the set where $x \geq 0$ and  $y = -x$.
The projection to the $x$--axis is locally open and is proper,
but it is not locally convex.

\item
Let $X$ be the union of the negative $x$--axis,
the portion of the $y$--axis where  $|y| < 1$,
and the set where  $x \geq 0$  and  $|y| = 1$.
The projection to the $x$--axis is locally convex
and is proper, but it is not locally open to its image.
\end{enumerate}
\end{Example}

\section{Spherical geometry}
\labell{sec:spherical geometry}

In this section we recall some elementary facts from spherical geometry.

Let $S^{n-1}$ be the unit sphere in $\R^n$.
A \emph{great circle} in $S^{n-1}$ is an intersection of $S^{n-1}$
with a two dimensional plane through the origin in $\R^n$.
A path $\gamma \co [0,1] \to S^{n-1}$ is a
\emph{weakly monotone geodesic}
if it is either constant or is equal to a composition
$$\begin{CD}
 [0,1] @>>> [\theta_0,\theta_1] @> (\cos(\cdot),\sin(\cdot)) >> S^1
 @> \iota >> S^{n-1}
\end{CD}$$
where the map $[0,1] \to [\theta_0,\theta_1]\subset \R$ is onto
and weakly monotone and where the map $\iota$ is an isometric
embedding of $S^1$ into $S^{n-1}$ as a great circle. The length
of the path is equal to
the length of the interval $[\theta_0,\theta_1]$. The path is
a weakly monotone \emph{short} geodesic if its length is $< \pi$.
A subset of $S^{n-1}$ is \emph{spherically convex} if every short
geodesic whose endpoints belong to the set is contained in the set.
An \emph{open hemisphere} is the intersection of $S^{n-1}$
with an open half--space whose boundary contains the origin.
An open hemisphere is spherically convex.
More generally, let $w$ be a point of $S^{n-1}$ 
and let $B(w, \eps)$ denote the open ball of radius $\eps$ centered at $w$. 
If $B(w,\eps) \cap S^{n-1}$ is contained in an open hemisphere,
then it is spherically convex.

\begin{Remark} \labell{notions of convexity}
There are several inequivalent notions of convexity in the literature;
see, e.g.,\cite[\S 9.1]{klee}.
``Robinson convexity" \cite{robinson} is defined for closed subsets 
of the sphere, and, for these subsets, it coincides with our notion 
of spherical convexity.
Note that this notion allows a set that consists of exactly 
one pair of antipodal points.
The definitions of ``geodesically convex" are stricter;
the entire sphere is spherically convex according to our definition 
but it is not geodesically convex.  
\end{Remark}

\begin{Lemma} \labell{straight}
Let $\gamma \co [0,1] \to \R^n$ be a weakly monotone straight path
that does not pass through the origin. Then
$ \gamma_S := \gamma / \| \gamma \| \co [0,1] \to S^{n-1} $
is a weakly monotone short geodesic.
\end{Lemma}

The geometric intuition behind the lemma should be clear.
We give an algebraic proof.

\begin{proof}
If $\gamma(\cdot)$ is contained in a line through the origin,
$\gamma_S(\cdot)$ is constant.

Suppose that $\gamma(\cdot)$ is not contained in a line
through the origin.
By choosing appropriate coordinates on the plane that contains
the set
$\{ 0 \} \cup \{ \gamma(t) \}$, we may assume that $n=2$
and that there exists $c>0$ such that $\gamma(\cdot)$
is contained in the horizontal line $\{ (\cdot,c) \}$.
Because $\gamma$ is weakly monotone straight,
there exist real numbers $a,b$ and a weakly monotone function
$s \co [0,1] \to [0,1]$ such that $s(0) = 0$, $s(1) = 1$, and
$ \gamma(t) = \left( \; (1-s(t))a + s(t)b \; , \; c \; \right) $.
Then
$ \gamma_S(t) = ( \cos \theta(t) , \sin \theta(t) ) $
where
\begin{equation} \labell{branch}
 \theta(t) \, = \, \text{Arctan} \frac{ c } { (1-s(t))a + s(t) b}
 \, \in \, (0,\pi).
\end{equation}
Because $t \mapsto s(t)$ is weakly monotone, so is the denominator
of~\eqref{branch}, and so is $t \mapsto \theta(t)$.
\end{proof}

We now give a converse result:

\begin{Lemma} \labell{straight2}
Let $\gamma_S \co [0,1] \to S^{n-1}$ be a weakly monotone
short geodesic. Let $\lambda_0$ and $\lambda_1$ be positive numbers.
Then there exists a weakly monotone straight
path $\gamma \co [0,1] \to \R^n$ such that
$\gamma(0) = \lambda_0 \gamma_S(0)$,
$\gamma(1) = \lambda_1 \gamma_S(1)$,
and $\gamma(t)/\| \gamma(t) \| = \gamma_S(t)$.
\end{Lemma}

\begin{proof}
If $\gamma_S(\cdot)$ is constant, set
$ \gamma(t) = (1-t) \lambda_0 \gamma_S(0)
                   + t \lambda_1 \gamma_S(1) $.

Suppose that $\gamma_S(\cdot)$ is not constant.
By choosing appropriate coordinates on the plane
that contains the set $\{ 0 \} \cup \{ \gamma_S(t) \}$,
we may assume that $n=2$ and that there exists a positive number $c$
such that $\lambda_0 \gamma_S(0)$ and $\lambda_1 \gamma_S(1)$
both belong to the horizontal line $\{ (\cdot , c) \}$.

Because $\gamma_S(\cdot)$ is a weakly monotone geodesic
that is contained in the upper half plane,
there exists a weakly monotone function
$ \theta \co [0,1] \to (0,\pi) $
such that
$ \gamma_S(t) = \left( \cos \theta(t) , \sin \theta(t) \right) $.
Project along radii to the line $\{ (\cdot,c) \}$: set
$$\gamma(t) = \left( c \cot\theta(t) , c \right) .$$
Then $t \mapsto \gamma(t)$ is weakly monotone straight,
and $\gamma / \| \gamma \| = \gamma_S$.
Also, $\gamma(0) = \lambda_0 \gamma_S(0)$
and $\gamma(1) = \lambda_1 \gamma_S(1)$.
\end{proof}

As before, $B(w,\eps)$ denotes the open ball of radius $\eps$
centered at $w$.

We recall an elementary fact that relates straightness
on spheres to straightness in vector spaces:

\begin{Lemma} \labell{radial projection}
There exists a homeomorphism $A$, from the upper hemisphere
$$ S_+^{n-1} = \{ (x_1, x_2,\dots,x_n) \in S^{n-1} \mid x_n > 0\} $$
onto $\R^{n-1}$, with the following properties.
Denote the north pole by $w_0$.
\begin{enumerate}
\item
For any $\eps > 0$, the map $A$ carries the ``cap"
$B(w_0,\eps) \cap S_+^{n-1}$
onto a ball centered at the origin in $\R^{n-1}$.
\item
Let $\gamma$ be a path in $S_+^{n-1}$. Then
$\gamma$ is a weakly monotone geodesic if and only if $A \circ
\gamma$ is a straight line segment in $\R^{n-1}$ with a weakly
monotone parametrization.
\end{enumerate}
\end{Lemma}

\begin{proof}
We set $A$ to be the projection from $S_+^{n-1}$ to the hyperplane
$\{ ( \cdot , \ldots , \cdot , 1 ) \}$
along rays emanating from the origin, followed by the projection
$\R^{n-1} \times \{ 1 \} \to \R^{n-1}$.
That is, $A(x_1,\ldots,x_{n-1},x_n)
 = ( x_1/x_n , \ldots , x_{n-1}/x_n )$.
This is a homeomorphism with inverse
$(y_1,\ldots,y_{n-1}) \mapsto
 \frac{1}{\sqrt{y_1^2+\ldots+y_{n-1}^2+1}} (y_1,\ldots,y_{n-1},1)$.
It maps the north pole $\omega_0$ to the origin, and a cap $B(w_0,\eps)
\cap S_+^{n-1}$ to a ball centered at the origin in $\R^{n-1}$.
If $P$ is a 2--plane through the origin in $\R^n$,
the map $A$ carries the great half--circle $P \cap S_+^{n-1}$
to the straight line $P \cap \{x_n=1\}$.

Every weakly monotone geodesic in $S^{n-1}_+$ has the form
$$ t \mapsto \left( \alpha \cos \theta(t) , \sin \theta(t) \right) $$
where $\alpha$ is in $S^{n-2}$ and $\theta \co [0,1] \to (0,\pi)$
is weakly monotone.  The map $A$ carries this geodesic to the path
$$ t \mapsto \alpha \cot \theta(t) .$$
Because the function $\cot(\cdot)$ is weakly monotone,
the path is weakly monotone straight.  Thus, the map $A$
carries a weakly monotone short geodesic in $S^{n-1}_+$
to a weakly monotone straight path in $\R^{n-1}$.
Conversely, the inverse map, $A^{-1}(x) = (x, 1) / \| (x, 1) \| $,
carries a weakly monotone straight path 
in $\R^{n-1}\cong \R^{n-1} \times \{ 1 \}$
to a weakly monotone short geodesic in $S^{n-1}_+$,
by Lemma~\ref{straight}.
\end{proof}

We now recall a property of spherical triangles:

\begin{Lemma} \labell{triangle}
In a spherical right triangle contained in a half--sphere $S^{n-1}_+$,
if one leg has length $ > \pi / 2$, then that leg
is longer than the hypotenuse.
\end{Lemma}

\begin{proof}
Given a spherical triangle in $S^{n-1}_+$, its vertices span a
three dimensional linear subspace; the intersection of this
subspace with $S^{n-1}_+$ is a two dimensional hemisphere that contains the
triangle. Hence it suffices to prove the lemma for a spherical triangle in
$S^2_+$.

Let $A, B, C$ denote the vertices of a spherical triangle in
$S^2_+$ and $a, b, c$ their facing arc lengths. Then $0< a, b, c<\pi$.
Suppose that the angle at $C$ is a right angle and that $a > \pi/2$.
The fundamental formula of spherical trigonometry
\cite[\S18.6]{berger} implies that
$\cos c = \cos a \cos b$, which implies that
$c \leq \pi/2 < a$ if $b \geq \pi/2$,
and $\pi/2 < c < a$ if $b<\pi/2$.
Therefore, the leg $a > \pi/2$ is always longer
than the hypotenuse $c$.
\end{proof}

\section{Spherically convex maps}
\labell{sec:spherically convex maps}

We now consider maps to spheres.
In this section, for simplicity, we restrict attention
to the sphere of radius one.
In later sections we take the freedom to use the same
terminology and results for spheres of arbitrary radii.

Recall that 
$B(w,\eps)$ denotes the open ball of radius $\eps$ centered at $w$.
Let $B_\eps$ be a shorthand for $B(0, \eps)$ and $\del B_\eps$ be its boundary.

\begin{Definition} \labell{spherically convex map}
A map $\psi$ from a topological space $X$ to the sphere $S^{n-1}$
is a \emph{spherically convex map}
if for every two points $x_0$ and $x_1$ in $X$,
if $\psi(x_1) \neq -\psi(x_0)$,
then there exists a path $\gamma \co [0,1] \to X$
such that $\gamma(0) = x_0$, $\gamma(1) = x_1$,
and the composition $\psi \circ \gamma \co [0,1] \to S^{n-1}$
is a weakly monotone short geodesic.
\end{Definition}

\begin{Remark} \labell{inclusion map}
A subset $X$ of $S^{n-1}$ is spherically convex
in the sense of Section~\ref{sec:spherical geometry}
exactly if the inclusion map
$\psi \co X \to S^{n-1}$ is a spherically convex map.
\end{Remark}

\begin{Remark} \labell{restrict}
Suppose that $\psi \co X \to S^{n-1}$ is spherically convex 
and let $A$ be a subset of $\R^n$.
If $A \cap S^{n-1}$ is spherically convex,
then the map $\psi|_{\psi\Inv(A)}$ is spherically convex.
If, additionally, $A \cap S^{n-1}$ does not consist of a single pair
of antipodal points, then $\psi\Inv(A)$ is connected.
In particular, if $H$ is an open half--space whose boundary contains the origin,
then $\psi|_{\psi\Inv(H)}$ is spherically convex
and $\psi\Inv(H)$ is connected.
More generally, if $w\in S^{n-1}$ and if $\eps>0$ is sufficiently small 
so that $B(w,\eps) \cap S^{n-1}$ is contained in an open hemisphere,
then $\psi|_{\psi\Inv(B(w,\eps))}$ is spherically convex
and $\psi\Inv(B(w,\eps))$ is connected.
\end{Remark}

We will adjust arguments of~\cite{karshon-marshall} to our needs.
To start, we have the following variant of Proposition~17
of~\cite{karshon-marshall}.

\begin{Lemma} \labell{good U}
Let $\varphi$ be a continuous map from a Hausdorff topological space
to $\R^n$.
Let $K$ be a compact connected subset of the level set $\varphi\Inv(0)$.
Suppose that each point $x$ in $K$ has an open
neighborhood $U_x$ such that the map
$\varphi|_{U_x} \co U_x \to \varphi(U_x)$
is convex and open.
Then there exists an open neighborhood $U_K$ of $K$
such that the map $ \varphi|_{U_K} \co {U_K} \to \varphi(U_K)$
is convex and open.

Moreover,
we can choose $U_K$ such that the following conditions hold.

\begin{itemize}
\item
Suppose that for each $x$ in $K$ and each sufficiently small $\delta > 0$, 
every point in
$\varphi\Inv(\del B_\delta) \cap U_x$ has a neighborhood $V \subset U_x$ 
such that
the restriction $\varphi|_{\varphi\Inv(\del B_\delta) \cap V}$
is spherically convex. Then we can choose $U_K$ so that,
for sufficiently small $\delta>0$,
every point in $\varphi\Inv(\del B_\delta) \cap U_K$
has a neighborhood $U \subset U_K$ such that
$\varphi|_{\varphi\Inv(\del B_\delta) \cap U}$
is also spherically convex.
\item
Suppose that for each $x$ in $K$ there exists
a cone $C_x$ with vertex at the origin
such that $\varphi(U_x)$ is an open subset of $C_x$.
Then the cones $C_x$ are all equal to each other, and we can choose $U_K$ 
so that $\varphi(U_K)$ is also an open subset of this common cone.
\end{itemize}
\end{Lemma}

\begin{Remark} \labell{vertex at origin}
A \emph{cone with vertex at the origin} is a set that is invariant 
under multiplication by positive numbers; cf.~\cite[\S 11.1.6]{bergerI}.
Some authors require a cone to also contain its vertex;
cf.\ \cite[II.(8.1)]{barvinok}. We only use this term for 
sets that already contain the origin.
\end{Remark}

\begin{proof}[Proof of Lemma \ref{good U}]
Let $U_1,\ldots,U_N$ be open sets such that
$$ U_i \cap K \neq \varnothing, $$
$$ K \subset U_1 \cup \ldots \cup U_N , $$
and
$$ \varphi|_{U_i} \co U_i \to \varphi(U_i)
\quad \text{ is convex and open.} $$

Suppose $U_i \cap U_j \cap K \neq \varnothing$.
Then $\varphi(U_i \cap U_j)$ contains $0$.
Since $U_i \cap U_j$ is open in $U_i$ and
the map $\varphi|_{U_i} \co U_i \to \varphi(U_i)$ is open,
there exists a positive number $\eps_{ij}$
such that $\varphi(U_i \cap U_j) \cap B_{\eps_{ij}}
 = \varphi(U_i) \cap B_{\eps_{ij}}$.

Let $\eps = \min \{ \eps_{ij} \ | \ U_i \cap U_j \cap K \neq \varnothing \}$.
Then $\varphi(U_i) \cap B_{\eps}
    = \varphi(U_j) \cap B_{\eps}
    = \varphi(U_i \cap U_j) \cap B_{\eps}$
whenever $U_i \cap U_j \cap K \neq \varnothing$.

Because $K$ is connected, every $U_i$ and $U_j$
can be connected by a sequence $U_i = U_{i_0},
U_{i_1}, \ldots, U_{i_s} = U_j$
such that $U_{i_{\ell-1}} \cap U_{i_\ell} \cap K \neq \varnothing$
for $\ell = 1,\ldots,s$.
So the set $\varphi(U_i) \cap B_{\eps}$
is the same for all $i$ and is equal
to $\varphi(U_i \cap U_j) \cap B_{\eps}$
whenever $U_i \cap U_j \cap K \neq \varnothing$.
Call this set $W$.
Let
$$U_K = (U_1 \cup \ldots \cup U_N) \cap \varphi\Inv(B_{\eps}).$$
Then
\begin{eqnarray*}
 \varphi(U_K) & = & \varphi(U_i) \cap B_{\eps}
                           \qquad \text{ for all }i \\
 & = & \varphi(U_i \cap U_j) \cap B_{\eps} \qquad
                      \text{ if } U_i \cap U_j \cap K \neq \varnothing \\
 & = & W.
\end{eqnarray*}

The level sets of $\varphi|_{U_K} \co U_K \to W$ are connected.
This follows from the following facts:
\begin{enumerate}
\item
For every $U_i$, every level set of $\varphi|_{U_K} \co U_K \to W$
meets $U_i$.
\item
Whenever $U_i \cap U_j \cap K \neq \varnothing$,
every level set of $\varphi|_{U_K} \co U_K \to W$
meets $U_i \cap U_j$.
\item
Every $U_i$ and $U_j$ can be connected by a sequence
$U_i = U_{i_0}, U_{i_1}, \ldots, U_{i_s} = U_j$
such that $U_{i_{\ell-1}} \cap U_{i_\ell} \cap K \neq \varnothing$.
\item
The level sets of each $\varphi|_{U_i}$ are connected.
\end{enumerate}

We now show that $\varphi|_{U_K}$ is convex.
Let $x_0$ and $x_1$ be points in $U_K$.
Then $x_0$ is contained in some $U_i$.
Because $\varphi(U_i)$ contains $\varphi(U_K)$,
there exists a point $\ol{x}_1$ in $U_i$
with $\varphi(\ol{x}_1) = \varphi(x_1)$;
because $\varphi|_{U_i}$ is convex, there exists a path
in $U_i$ from $x_0$ to $\ol{x}_1$ whose composition with $\varphi$ is 
weakly monotone straight and
whose image in $\R^n$ is $[\varphi(x_0),\varphi(x_1)]$.
Because the level sets of $\varphi|_{U_K}$ are connected,
there exists a path from $\ol{x}_1$ to $x_1$
that is contained in $\varphi\Inv (\varphi(x_1)) $;
the concatenation of the path from $x_0$ to $\ol{x}_1$
with the path from $\ol{x}_1$ to $x_1$
is a path from $x_0$ to $x_1$ whose composition with $\varphi$ 
is weakly monotone straight
with image
$[\varphi(x_0),\varphi(x_1)]$.
This shows that $\varphi|_{U_K}$ is convex.

Since the map
$\varphi|_{U_i \cap \varphi\Inv(B_{\eps})} \co
 U_i \cap \varphi\Inv(B_{\eps}) \to W$
is open for every $i$,
the map $\varphi|_{U_K} \co U_K \to W$ is open.

Suppose that $\delta_i > 0 $ is such that,
for all $0 < \delta < \delta_i$, every point 
in $\varphi\Inv(\del B_\delta) \cap U_i$ has a neighborhood $V \subset U_i$ 
such that
the restriction $\varphi|_{\varphi\Inv(\del B_\delta) \cap V}$
is spherically convex.
Then, for every $0 < \delta < \min \{ \eps, \delta_1, \ldots, \delta_N \}$,
because $U_K \cap \varphi\Inv(\del B_\delta)$
is the union of the sets $U_i \cap \varphi\Inv(\del B_\delta)$,
every point in $U_K \cap \varphi\Inv(\del B_\delta)$
has a neighborhood $U$ (namely, $U=V\cap \varphi\Inv(B_\eps)$ 
for some $V \subset U_i$)
such that the restriction of $\varphi$
to $U \cap \varphi\Inv(\del B_\delta)$ is spherically convex.

Suppose that $\varphi(U_i)$ is an open subset of a cone
with vertex at the origin.
Then, because $\varphi(U_i) \cap B_\eps = W$,
this cone is equal to $\Rplus \cdot W$, and $W$ is open in the cone.
In particular, the cone is independent of $i$.
Then $\varphi(U_K)$, being also equal to $W$,
is an open subset of this cone.
\end{proof}

\begin{Notation} \labell{notation:component}
Fix a map $\psi \co X \to S^{n-1}$.
For $x \in X$ with $\psi(x) = w$, we denote by $[x]$
the connected component of $x$ in $\psi\Inv(w)$,
and we denote by $U_{[x],\eps}$ the connected component of $x$
in $\psi\Inv(B(w,\eps))$.
\end{Notation}

We now give another variant of Proposition~17 of~\cite{karshon-marshall},
which applies to a spherical map:

\begin{Lemma} \labell{spherical intgoalB}
Let $X$ be a compact Hausdorff topological space, and let
$\psi \co X \to S^{n-1}$ be a continuous map.
Suppose that every point in $X$ is contained in an open set $U$
such that the map $\psi|_U \co U \to \Phi(U)$ is open
and is spherically convex. Then for every point $x$ in $X$
and every sufficiently small $\eps > 0$, the map
$\psi|_{U_{[x],\eps}} \co U_{[x],\eps} \to \psi(U_{[x],\eps})$
is open and spherically convex.
\end{Lemma}

\begin{proof}
Fix $x \in X$.  Without loss of generality, assume that
$w := \psi(x)$ is the north pole. 
Let $S_+^{n-1}$ denote the open upper hemisphere, 
let $A \co S_+^{n-1} \to \R^{n-1}$ 
be the homeomorphism of Lemma \ref{radial projection},
and let $X_+ = \psi\Inv(S_+^{n-1})$.
Using Lemma~\ref{radial projection},
and applying Lemma~\ref{good U} 
to the composition
$A \circ \psi|_{X_+} \co X_+ \to \R^{n-1}$
and the set $K = [x]$, we find a neighborhood $V$ of $[x]$
that is contained in $X_+$
and such that the map $\psi \co V \to \psi(V)$
is open and spherically convex.

We now show that there exists $\eps > 0$ such that the
neighborhood $V$ of $[x]$ contains the connected component $U_{[x],\eps}$
of $[x]$ in $\psi\Inv(B(w,\eps))$.

Because $\psi$ is proper, the level set $\psi\Inv(w)$ is compact.
Because every point in $X$ has a neighborhood on which $\psi$
(is spherically convex, hence) has connected level sets,
the level set $\psi\Inv(w)$
is locally connected.  These two properties imply that the level set
$\psi\Inv(w)$ has only finitely many connected components.
Because these components are compact and disjoint and $X$ is Hausdorff,
there exist disjoint open subsets $\calO_1,\ldots,\calO_k$ of $X$
such that every $\calO_j$ contains exactly one component of $\psi\Inv(w)$.
Without loss of generality, suppose that the component $[x]$ is contained
in the set $\calO_1$. Let $\calO_1' = \calO_1 \cap V$.
Because $\psi$ is proper and $\calO_1' \cup \calO_2 \cup \ldots \calO_k$
is a neighborhood of the level set $\psi\Inv(w)$, there exists
$\eps > 0$ such that $\psi\Inv(B(w,\eps))$ is contained
in $\calO_1' \cup \calO_2 \cup \ldots \calO_k$.
Because the sets $\calO_1',\calO_2,\ldots,\calO_k$
are open and disjoint and the set $U_{[x],\eps}$ is connected
and meets ($[x]$, hence) $\calO_1'$, the set $U_{[x],\eps}$
is entirely contained in $\calO_1'$.
In particular, $U_{[x],\eps}$ is contained in $V$.
Choosing $\eps$ sufficiently small, we may assume that $B(w,\eps) \cap S^{n-1}$
is contained in an open hemisphere.

Because $\psi|_V$ is spherically convex, by Remark~\ref{restrict},
$V \cap \psi\Inv (B(w,\eps))$ is connected.
Since $V \cap \psi\Inv (B(w,\eps))$ is a connected subset
of $\psi\Inv (B(w,\eps))$ that contains the connected component
$U_{[x],\eps}$, it must be equal to this component:
$V \cap \psi\Inv (B(w,\eps)) = U_{[x],\eps}$.
By this and Remark~\ref{restrict}, since $\psi|_V$ is spherically convex
and open to its image, so is $\psi|_{U_{[x],\eps}}$.
\end{proof}

\begin{Remark} \labell{rk metric}
A similar result holds with more general target spaces;
one can avoid Lemma~\ref{radial projection}
and work with an analogue of Lemma~\ref{good U}
for more general target spaces.
See \cite{BOR2}.  
We will not need this level of generality.
\end{Remark}

We let $l(\gamma) \in [0,\infty]$ denote the length of a continuous curve
$\gamma \co [0,1] \to \R^n$.

We now have the following variant
of Theorem~15 of~\cite{karshon-marshall}.

\begin{Proposition} \labell{exists monotone geodesic}
Let $X$ be a compact connected Hausdorff topological space and let
$$\psi \co X \to S^{n-1}$$
be a continuous map.  Suppose that every point of $X$ is contained
in an open set $U \subset X$ such that the map $\psi|_U \co
U \to \psi(U)$ is open and is spherically convex. Then every
two points in $X$ can be connected by a path $\gamma$ such that
$\psi \circ \gamma$ is a weakly monotone geodesic
and such that $l(\psi \circ \gamma) \leq l(\psi \circ \gamma')$
for every other path $\gamma'$ connecting the two points.
\end{Proposition}

The following proof of Proposition \ref{exists monotone geodesic}
is analogous to those of Theorem~15 and Lemma~23 of~\cite{karshon-marshall},
which follow Condevaux--Dazord--Molino \cite{CDM}.

\begin{proof}[Proof of Proposition \ref{exists monotone geodesic}.]
For two points $x$ and $x'$ in $X$, define
$d_\psi (x,x')$ to be the infimum of the lengths
$l(\psi \circ \gamma)$ as $\gamma$ varies over
all paths in $X$ from $x$ to $x'$.

Note that if $x' \in U_{[x],\eps/2}$ and $d_\psi (x',x'') \leq \eps/2$,
then $x'' \in U_{[x],\eps}$ (cf.\ Notation~\ref{notation:component}).

Let $x_0$ and $x_1$ be points of $X$.  We show that
there exists a ``midpoint" $x_{1/2}$ such that
$$ d_\psi(x_0,x_\half) = d_\psi(x_\half,x_1)
    = \half d_\psi(x_0,x_1).$$
Choose paths $\gamma_n$
connecting $x_0$ to $x_1$ such that the sequence of lengths
$l(\psi \circ \gamma_k)$ converges to $d_\psi(x_0,x_1)$. Let
$t_k \in [0,1]$ be such that
$l(\psi \circ \gamma_k|_{[0,t_k]}) =
 l(\psi \circ \gamma_k|_{[t_k,1]})
 = \frac{1}{2} l(\psi \circ \gamma_k)$. Since $X$ is compact,
there exists a point $x_{1/2}$ such that every neighborhood of $x_{1/2}$
contains $\gamma_k(t_k)$ for infinitely many values of $k$.
Let $\eps > 0$. There exists $j$
such that $\gamma_j(t_j)$ belongs to a
neighborhood $U$ of $x_{1/2}$ on which $\psi$ is spherically convex,
such that $d_\psi(\gamma_j(t_j),x_{1/2}) < \eps/2$, and such that
$l(\psi \circ \gamma_j) < d_{\psi}(x_0,x_1) + \eps$.
The path $\gamma_j|_{[0,t_j]}$, followed by a path in $U$ from
$\gamma_j(t_j)$ to $x_{1/2}$ whose composition with $\psi$ is a
weakly monotone geodesic,
form a path from $x_0$ to $x_{1/2}$
whose composition with $\psi$ has length $<
\frac{1}{2}d_\psi(x_0,x_1) + \eps$. This implies that
$d_{\psi}(x_0, x_{1/2})\leq \frac{1}{2}d_\psi(x_0,x_1)$.
Similarly, $d_{\psi}(x_{1/2},x_1)\leq
\frac{1}{2}d_\psi(x_0,x_1)$.

Iterating, we find points $x_{j/{2^m}}$,
for $0 \leq j \leq 2^m$,
such that
\begin{equation} \labell{equal}
 d_\psi \left( x_{\frac{j_1}{2^m}} , x_{\frac{j_2}{2^m}} \right)
 = \frac{ |j_1 - j_2| }{ 2^m }
\quad \text{ for all } j_1, j_2 \in \{ 0, 1, \ldots, 2^m \}.
\end{equation}

By Lemma~\ref{spherical intgoalB},
and since $X$ is compact,
there exist points $\xbar_1,\ldots,\xbar_N$
and positive numbers $\eps_1,\ldots,\eps_N$
such that $\psi|_{U_{[\xbar_i],\eps_i}}$
is open to its image and is spherically convex
and such that the sets $U_{[\xbar_i],\eps_i/2}$ cover $X$.
Let $\eps = \min \{ \eps_1, \ldots, \eps_N \}$.
Then for every $x'$ and $x''$, if $d_\psi(x',x'') \leq \eps/2$,
then there exists a path from $x'$ to $x''$
whose composition with $\psi$ is a weakly monotone short geodesic.

Choose $m$ large enough so that
$\frac{1}{2^m} d_\psi(x_0,x_1) < \frac{1}{4} \eps$.
Then each pair among the three points $x_{(j-1)/2^m}, x_{j/2^m}, x_{(j+1)/2^m}$ 
can be connected by a path whose composition with $\psi$
is a weakly monotone short geodesic.
Because
$d_\psi(x_{(j-1)/2^m},x_{j/2^m})+d_\psi(x_{j/2^m},x_{(j+1)/2^m})=d_\psi(x_{(j-1)/2^m},x_{(j+1)/2^m})$,
these paths fit into a path from $x_0$ to $x_1$ whose composition
with $\psi$ is a weakly monotone geodesic.
Because the length of this geodesic is $d_\psi(x_0,x_1)$,
it is shorter than $\psi \circ \gamma'$ for any other path
$\gamma'$ connecting $x_0$ to $x_1$.
\end{proof}

We have this easy consequence
of Proposition~\ref{exists monotone geodesic}:

\begin{Corollary}
Under the assumptions of Proposition~\ref{exists monotone geodesic},
let $x_0$ and $x_1$ be points of $X$, and suppose that there exists
a path $\tgamma$ in $X$ from $x_0$ to $x_1$
such that $l(\psi \circ \tgamma) < \pi$.
Then there exists a path $\gamma$ in $X$ from $x_0$ to $x_1$
such that $\psi \circ \gamma$ is a short geodesic.
\end{Corollary}

\begin{Remark}
Let $\psi \co X \to S^{n-1}$ be as in
Proposition~\ref{exists monotone geodesic}.
As in Condevaux--Dazord--Molino \cite{CDM}, consider the equivalence
relation on $X$ in which $x_0 \sim x_1$ 
iff there exists a path $\gamma \co [0,1] \to X$
such that $\gamma(0) = x_0$ and $\gamma(1) = x_1$
and such that $\psi \circ \gamma$ is constant.
The function $d_\psi$ defined in the proof 
of Proposition~\ref{exists monotone geodesic} descends
to a metric on $X/\!\!\sim$;
this follows from Lemma~\ref{spherical intgoalB}.
With this metric, $X/\!\!\sim$ is a \emph{length space}.
(This means 
that the distance between two points is equal to the infimum
of the lengths of curves that connect them. As usual, the length of 
a continuous curve $\{\gamma(t)\}_{t \in [0,1]}$ is the supremum 
of $\sum_{i=1}^n \text{distance}(\gamma(t_{i}),\gamma(t_{i-1}))$
over all partitions $0=t_0<t_1<\ldots<t_n=1$.)
Proposition~\ref{exists monotone geodesic}
then follows from the Hopf--Rinow theorem for length spaces.
See~\cite[Chapter~I.3]{bridson-haefliger}.
\end{Remark}

\section{Shortening a nonminimal monotone geodesic}
\labell{sec:shorten}

\begin{Proposition} \labell{shorten}
Let $X$ be a compact connected Hausdorff space.
Let $\psi \co X \to S^{n-1}$ be a continuous map.
Suppose that each point in $X$ is contained in an open set $U \subset X$
such that $\psi|_U \co U \to \psi(U)$
is spherically convex and is open.

Let $\gamma \co [0,1] \to X$ be a path in $X$
such that $\psi \circ \gamma$ is a weakly monotone geodesic
of length $ \pi < l(\psi \circ \gamma) < 2\pi $.
Let $E \subset S^{n-1}$ be the great circle that contains
this geodesic; suppose that the image $\psi(X)$
is not entirely contained in $E$.
Then there exists a path $\tgamma \co [0,1] \to X$,
with the same endpoints as $\gamma$, such that
$l(\psi \circ \tgamma) < l(\psi \circ \gamma)$.
\end{Proposition}

\begin{proof}
Let $t_{\half}$ be such that
$l(\psi \circ \gamma|_{[0,t_\half]})
 = l(\psi \circ \gamma|_{[t_\half,1]})
 = \half l(\psi \circ \gamma).$
Let $x_0 = \gamma(0)$, $x_\half = \gamma(t_\half)$, $x_1 = \gamma(1)$.
Because $\psi \circ \gamma|_{[0,t_\half]}$
and $\psi \circ \gamma|_{[t_\half,1]}$ are short geodesics,
$d_\psi(x_0,x_\half) = d_\psi(x_\half,x_1)
 = \half l(\psi \circ \gamma)$.

By Proposition \ref{exists monotone geodesic} 
and because the image $\Psi(X)$ is not contained in $E$,
we can connect $x_\half$
to a point outside $\psi\Inv(E)$ by a path $\gammahat$
whose composition with $\psi$ is a weakly monotone geodesic.

For each $\eps > 0$ sufficiently small, let $\xhat_\eps$ be a point on the path $\gammahat$
such that $d_\psi(x_\half,\xhat_\eps) = \eps$, and let $x_{t_\eps}$
be a point on the path $\gamma$ such that $\psi(x_{t_\eps})$
is closest to $\psi(\xhat_\eps)$ among all points on the great
circle $E$.  Then there is a path in $X$ from $x_{t_\eps}$
to $\xhat_\eps$ whose composition with $\psi$ is contained
in a great circle that is perpendicular to $\gamma$.

We have $\lim\limits_{\eps \to 0} d_\psi(x_{t_\half},x_{t_\eps})
 = \lim\limits_{\eps \to 0} d_\psi(x_{t_\half},\xhat_\eps) = 0$.

Fix $\eps > 0$ sufficiently small so that
$$ d_\psi(x_0,x_{t_\eps}) > \frac{\pi}{2}, \quad
   d_\psi(x_1,x_{t_\eps}) > \frac{\pi}{2}, \quad
\text{and} \quad
   d_\psi(x_1,x_{t_\eps}) + d_\psi(x_{t_\eps}, \xhat_\eps) < \pi .$$

By Lemma \ref{triangle},
$d_\psi(x_0,\xhat_\eps) < d_\psi(x_0,x_{t_\eps})$
and
$d_\psi(x_1,\xhat_\eps) < d_\psi(x_1,x_{t_\eps})$.
Thus, there exists a path from $x_0$ to $\xhat_\eps$
whose composition with $\psi$ has length $< d_\psi(x_0,x_{t_\eps})$,
and there exists a path from $\xhat_\eps$ to $x_1$ whose composition
with $\psi$ has length $< d_\psi(x_{t_\eps},x_1)$.
The concatenation of these paths is a path from $x_0$ to $x_1$
whose composition with $\psi$ has length
$ < d_\psi(x_0,x_{t_\eps}) + d_\psi(x_{t_\eps},x_1)
  = l(\psi\circ \gamma)$.
\end{proof}

\begin{Corollary} \labell{shortening}
Let $X$ be a compact connected Hausdorff topological space. 
Let $\psi \co X \to S^{n-1}$ be a continuous map. 
Suppose that each point in $X$ has an open neighborhood $U$ 
such that $\psi|_U$ is spherically convex and is open 
as a map to its image, $\psi(U)$. 
Suppose that $\psi(X)$ is not contained in a great circle. 
Then the map $\psi \co X \to \psi(X) \subset S^{n-1}$ 
is spherically convex and is open as a map to its image.
\end{Corollary}

\begin{proof}
Let $x_0$ and $x_1$ be points of $X$
such that $\psi(x_1) \neq -\psi(x_0)$.
By Proposition \ref{exists monotone geodesic}, 
$x_0$ and $x_1$ can be connected by a weakly monotone geodesic 
that is shortest among all paths from $x_0$ to $x_1$.
Since $\psi(X)$ is not
contained in a great circle, by Proposition \ref{shorten}, 
this geodesic cannot have length $> \pi$.
So it has length $\leq \pi$.  
Because $\psi(x_1) \neq -\psi(x_0)$,
it has length $<\pi$. 
That is, $x_0$ and $x_1$ can be connected by a weakly monotone
short geodesic.
This proves that $\psi$ is spherically convex.

Fix $w \in S^{n-1}$.
Because $\psi$ is spherically convex,
if $\eps$ is sufficiently small, 
$\psi\Inv(B(w,\eps))$ is connected (cf.\ Remark~\ref{restrict}).
So $\psi\Inv(B(w,\eps))$ is equal to
$U_{[x],\eps}$ for any $x \in \psi\Inv(w)$
(cf.~Notation~\ref{notation:component}).
By Lemma~\ref{spherical intgoalB},
if $\eps$ is sufficiently small, it follows that
$\psi|_{\psi\Inv(B(w,\eps))}$ is open as a map to its image.
Since for every $w \in S^{n-1}$ there exists $\eps>0$ such that 
$\psi|_{\psi\Inv(B(w,\eps))}$ is open as a map to its image,
$\psi$ is open as a map to its image.
\end{proof}

We now give the ultimate ``local to global" result,
for maps to spheres:

\begin{Proposition} \labell{two possibilities}
Let $X$ be a compact connected Hausdorff topological space. 
Let $\psi \co X \to S^{n-1}$ be a continuous map. 
Suppose that every point in $X$ has an open neighborhood $U$ 
such that $\psi|_U$ is spherically convex 
and is open as a map to its image, $\psi(U)$. 


Suppose that $\psi(X)$ is not contained in a great circle. 
Then all the following results hold.
\begin{enumerate}
\item
For every $x_0$ and $x_1$ in $X$
there exists a path $\gamma$ from $x_0$ to $x_1$
such that $\psi \circ \gamma$ is a weakly monotone geodesic
of length $\leq \pi$.
\item
The set $C = \R_{\geq 0} \cdot \psi(X)$ is a convex cone 
with vertex at the origin:
if $w_1, w_2 \in C$ and $\lambda_1, \lambda_2 \in \R_{\geq 0}$,
then $\lambda_1 w_1 + \lambda_2 w_2 \in C$. 
\item
The level sets of $\psi$ are connected.
\item
The map $\psi$ is open as a map to its image.
\end{enumerate}

Suppose that the image $\psi(X)$ is contained in a great circle.
Let $\iota \co S^1 \to S^{n-1}$ be an isometric parametrization
of that great circle.
Then exactly one of the following three possibilities occurs.

\begin{enumerate}
	\renewcommand{\theenumi}{\roman{enumi}}
\item
There exist an interval $[\theta_0,\theta_1] \subset \R$,
and a surjective map
$ \tpsi \co X \to [\theta_0,\theta_1] $
which is convex and open, such that the map
$\psi \co X \to S^{n-1}$
is equal to the composition
$$ \begin{CD}
 X @> \tpsi >> [\theta_0,\theta_1]
  @> (\cos(\cdot),\sin(\cdot)) >> S^1 @> \iota >> S^{n-1}.
\end{CD} $$

\item
There exists a surjective map
$\tpsi \co X \to S^1$,
which is spherically convex and open, and a positive integer $m$,
such that the map $\psi$ is equal to the composition
$$ \begin{CD}
 X @> \tpsi >> S^1 @>
(\cos\theta,\sin\theta) \mapsto (\cos m\theta,\sin m\theta)
>> S^1 @> \iota >> S^{n-1}.
\end{CD} $$

\item
$\psi$ is constant.
\end{enumerate}
\end{Proposition}

\begin{Remark} \labell{rk:BOR2}
The case 
of Proposition~\ref{two possibilities}
in which $\psi(X)$ is not contained in a great circle
follows from Theorem~2.17 of~\cite{BOR2} 
with two slight adjustments.
In Definition~2.9 of~\cite{BOR2} 
of \emph{local convexity data},
replace ``for every sufficiently small neighborhood $U_x$ of $x$
the set $f(U_x)$ is convex"
by ``there exist arbitrarily small neighborhoods $U_x$ of $x$
such that the set $f(U_x)$ is convex".
In part (ii) of Theorem~2.17 of~\cite{BOR2}, apply 
the ``uniquely geodesic" assumption to~$f(X)$, not to~$Y$. 
\end{Remark}

\begin{proof}[Proof of Proposition~\ref{two possibilities}]
We first analyze the case that the image $\psi(X)$ is not contained 
in a great circle.

Let $x_0$ and $x_1$ be points of $X$.
If $\psi(x_0) \neq -\psi(x_1)$, Part (1) follows from
Corollary~\ref{shortening}.
If $\psi(x_0) = -\psi(x_1)$, take any $x'$ such that
$\psi(x')$ is different from $\psi(x_0)$ and $\psi(x_1)$.
By Corollary~\ref{shortening}, connect $x_0$ to $x'$ and $x'$ to $x_1$
by paths whose images are short geodesics.
The concatenation of these paths is a path from $x_0$ to $x_1$
whose composition with $\psi$ is a weakly monotone geodesic
of length $\pi$.  This gives part (1).

Part~(2) follows from Part~(1) and Lemma~\ref{straight2}.

Part~(3) follows from Part~(1). For any two points $x_0, x_1$ in the
same level set, there exists a path $\gamma$ in $X$ connecting
them such that $\psi \circ \gamma$ is a weakly monotone short geodesic.
Because $\psi(x_0)=\psi(x_1)$, the composition $\psi \circ \gamma$ 
must be constant.

Part~(4) was proved in Corollary~\ref{shortening}.

\bigskip

We now analyze the case that the image $\psi(X)$ 
\emph{is} contained in a great circle.
Without loss of generality, assume that $n=2$ and $\iota = $identity.
By the theory of covering spaces, exactly one of the following possibilities
occurs.
\begin{itemize}
\item[(a)]
The image of $\pi_1(X)$ in $\pi_1(S^1)$ is trivial,
and there exists a map $\tpsi \co X \to \R$
such that the map $\psi \co X \to S^1$ is equal to the composition
$$ \begin{CD}
 X @> \tpsi >> \R @> (\cos(\cdot),\sin(\cdot))  >> S^1 .
\end{CD} $$
\item[(b)]
The image of $\pi_1(X)$ in $\pi_1(S^1) \cong \Z$ is the subgroup
of index $m$, and there exists a map $\psi \co X \to S^1$
such that the map $\psi$ is equal to the composition
$$ \begin{CD}
 X @> \tpsi >> S^1 @>
 (\cos\theta,\sin\theta) \mapsto (\cos m\theta,\sin m\theta) >> S^1
\end{CD} $$
and such that the map $\tpsi_* \co \pi_1(X) \to \pi_1(S^1)$ is onto.
\end{itemize}

\medskip

Assume that we are in case (a).

Let $x$ be a point of $X$.
Let $U$ be a neighborhood of $x$ in $X$
such that $\psi|_U$ is spherically convex and is open to its image.

Let $J \subset S^1$ be a half--circle that contains $\psi(x)$.
Then $U' := U \cap \psi\Inv(J)$ is also a neighborhood of $x$
on which $\psi$ is spherically convex and is open to its image,
and $U'$ is connected.

The preimage of $J$ under the map $(\cos(\cdot),\sin(\cdot))$
is a disjoint union of segments.
Because $U'$ is connected, $\tpsi(U')$ is contained in one
of these segments; call this segment $\tJ$.

The map $(\cos(\cdot),\sin(\cdot))$ restricts to a homeomorphism
from $\tJ$ onto $J$.  The map $\psi|_{U'}$ is the composition
of $\tpsi|_{U'}$ with this homeomorphism.
Therefore, because $\psi|_{U'}$ is open to its image,
so is $\tpsi|_{U'}$.

A lifting to $\R$ of a weakly monotone geodesic in $S^1$
is weakly monotone.
Therefore, because $\psi|_{U'}$ is spherically convex,
$\tpsi|_{U'}$ is convex.

We have shown that every point in $X$ has a neighborhood $U'$
such that $\tpsi|_{U'}$ is convex and is open to its image.
By Theorem~\ref{local-global} it follows
that $\tpsi \co X \to \R$
is convex and is open to its image.
Because $X$ is compact, the image of $\tpsi$
is either a single point or a closed segment.
This shows that exactly one of the possibilities (i) or (iii)
must occur.

\medskip

Now assume that we are in case (b).

Let $x$ be a point of $X$.
Let $U$ be a neighborhood of $x$ in $X$
such that $\psi|_U$ is spherically convex and is open to its image.

Let $J \subset S^1$ be an open arc of the circle
that contains $\psi(x)$ and that has length $<2\pi/m$.
Then $U' := U \cap \psi\Inv(J)$ is also a neighborhood of $x$
on which $\psi$ is spherically convex and is open to its image,
and $U'$ is connected.

The preimage of $J$ under the map
$(\cos\theta,\sin\theta) \mapsto (\cos m\theta,\sin m\theta)$
is a disjoint union of $m$ arcs of $S^1$.
Because $U'$ is connected, $\tpsi(U')$ is contained in one
of these arcs; call this arc $\tJ$.

The map $(\cos(\cdot),\sin(\cdot))$ restricts to a homeomorphism
from $\tJ$ onto $J$.  The map $\psi|_{U'}$ is the composition
of $\tpsi|_{U'}$ with this homeomorphism.
Therefore, because $\psi|_{U'}$ is open to its image,
so is $\tpsi|_{U'}$.

A lifting to $\tJ$ of a weakly monotone geodesic in $J$
is a weakly monotone geodesic in $\tJ$.
Therefore, because $\psi|_{U'}$ is spherically convex,
so is $\tpsi|_{U'}$.

We have shown that every point in $X$ has a neighborhood $U'$
such that $\tpsi|_{U'}$ is spherically convex and is open to its image.
Also, $\tpsi$ induces a surjection $\pi_1(X) \to \pi_1(S^1)$.
It remains to show that these assumptions imply
that $\tpsi\co X \to S^1$ is spherically convex and open.

We have a commuting diagram
$$ \xymatrix{
 \widetilde{X} \ar[r]^{\tpsi'} \ar[d]_{\pi}
   & \R \ar[d]^{(\cos(\cdot),\sin(\cdot))} \\
   X  \ar[r]_{\tpsi}               & S^1
}$$
where $\tilde{X}$ is the fibered product $X \times_{S^1} \R$
and $\pi \co \tilde{X} \to X$ is the covering map.
An argument similar to that of case (a) applied to $\tpsi \circ \pi$
show that every point in $\tilde{X}$ has a neighborhood
on which $\tpsi'$ is open to its image and is convex.
The map $\tpsi' \co \tilde{X} \to \R$ is proper;
this follows from the fact that $X$ is compact.
The space $\widetilde{X}$ is connected; this follows from the assumptions
that the map $\tpsi_* \co \pi_1(X) \to \pi_1(S^1)$ is onto
and the map $\tpsi \co X \to S^1$ is spherically convex.
By Theorem~\ref{local-global} it follows that the map $\tpsi'$
is open to its image and is convex.
Because we are in case (b), the map $\tpsi'$ is onto.
Hence $\tpsi\co X \to S^1$ is spherically convex and open.
\end{proof}

\section{Local to global convexity for conification of nonvanishing
functions}
\labell{sec:nonvanishing}

\begin{Lemma}\labell{slide'}
Let $X$ be a Hausdorff topological space
and $\varphi \co X \to \R^n$ a continuous map.
Define a map $\Psi \co X \times \Rplus \to \R^n$
by $\Psi(x,\lambda) = \lambda \varphi(x)$.
Suppose that the map $\varphi$ is convex.
Then the map $\Psi$ is convex.
\end{Lemma}

\begin{proof}
Let $(x_0, \lambda_0)$ and $(x_1, \lambda_1)$ be two points 
in $X \times \Rplus$.
Because the map $\varphi$ is convex,
there exists a path $x(t)$ in $X$ from $x_0$ to $x_1$
such that $\varphi(x(t))$ is weakly monotone straight.
That is, there exists a weakly monotone continuous function
$$ s \co [0,1] \to [0,1] $$
such that $s(0)=0$ and $s(1)=1$ and such that
$$ \varphi(x(t)) =
 (1 - s(t)) \varphi(x_0) + s(t) \varphi(x_1) .$$
We rewrite the right hand side as
$$ \frac{1 - s(t)}{\lambda_0} \lambda_0 \varphi(x_0)
 + \frac{ s(t) }{\lambda_1} \lambda_1 \varphi(x_1) ,$$
and then divide both sides of the equation by the sum of the coefficients,
$$ \frac{1-s(t)}{\lambda_0} \, + \, \frac{s(t)}{\lambda_1} ,$$
which is positive.
Setting
$$ \lambda(t) \ = \ \frac{1}{\displaystyle{
             \frac{1-s(t)}{\lambda_0} + \frac{s(t)}{\lambda_1}} } $$
and
\begin{equation} \labell{tilde s}
 \ts(t) \ = \ \frac{ \displaystyle{ \frac{s(t)}{\lambda_1} } }
            { \displaystyle{ \frac{1-s(t)}{\lambda_0}
                           + \frac{s(t)}{\lambda_1}  } }  \ ,
\end{equation}
we get
\begin{equation} \labell{path}
 \lambda(t) \varphi(x(t)) \ =  \
   ( 1 - \ts(t) ) \, \lambda_0 \, \varphi(x_0)
        \ + \ \ts(t) \,  \lambda_1 \, \varphi(x_1) \ .
\end{equation}
Because $\lambda(t)$ is a path of positive numbers connecting
$\lambda_0$ to $\lambda_1$,
the path $y(t)=(x(t),\lambda(t))$ in $X \times \Rplus$
connects $(x_0,\lambda_0)$ to $(x_1,\lambda_1)$.
It remains to show that 
$\Psi(y(t))$ is weakly monotone straight. By \eqref{path}, 
it suffices to show that $\ts(t)$ is a weakly monotone continuous function
from $[0,1]$ to $[0,1]$ such that $\ts(0)=0$ and $\ts(1)=1$.
Because $s(\cdot)$ is a continuous function
from $[0,1]$ to $[0,1]$ such that $s(0) = 0$ and $s(1) = 1$,
and by~\eqref{tilde s}, we see that $\ts(\cdot)$ is a continuous function
from $[0,1]$ to $[0,1]$ such that $\ts(0) = 0$ and $\ts(1) = 1$.
Monotonicity of $\ts(t)$ follows from that of $s(t)$
since, whenever $s(t) > 0$
$$ \frac{1}{\ts(t)} = 1
   + \frac{\lambda_1 }{\lambda_0}
     \left( \frac{1}{s(t)} - 1\right) .$$
\end{proof}

\begin{Lemma} \labell{up-down}
Let $X$ be a Hausdorff topological space and
$\varphi \co X \to \R^n$ a continuous nonvanishing map.
Define maps $\Psi \co X \times \Rplus \to \R^n$
and $\Psibar \co X \times \Rplus \to S^{n-1}$
by $\Psi(x,\lambda) = \lambda \varphi(x)$
and $\Psibar = \Psi / \| \Psi \|$.
Also define $\psi = \varphi / \| \varphi \| \co X \to S^{n-1}$.

Let $U$ be an open subset of $X \times \Rplus$.
Let $V$ be the image of $U$ under the projection $X \times \Rplus \to X$.

The set $V$ is an open subset of $X$.
If the map $\Psi|_U$ is convex, then the map $\psi|_V$ is spherically convex.
If the map $\Psibar|_U$ is open to its image, so is the map $\psi|_V$.
\end{Lemma}

\begin{proof}
Openness of $V$ in $X$ follows from the definition 
of the product topology on $X \times \Rplus$.

The set $\Psibar(U)$ is equal to the set $\psi(V)$; call this set $W$.
We have a commuting diagram of continuous maps:
$$ \xymatrix{
X \times \Rplus & \supset & U \ar[d]_{\text{projection}} \ar[rd]^{\Psibar|_U} 
  & & \\
X  & \supset & V \ar[r]_{\psi|_V} & W & \subset S^{n-1}.
}$$

Suppose that the map $\Psi|_U \co U \to \R^n$ is convex.
By Lemma~\ref{straight}, the map $\Psibar|_U \co U \to W$
is spherically convex.
Because the projection map $U \to V$ is onto,
and by the commuting diagram,
the map $\psi|_V \co V \to W$ is spherically convex.

Suppose that the map $\Psibar|_U \co U \to W$ is open.
Because the projection map $U \to V$ is onto,
and by the commuting diagram,
the map $\psi|_V \co V \to W$ is open.
\end{proof}

\begin{Corollary} \labell{every}
Let $X$ be a Hausdorff topological space and
$\varphi \co X \to \R^n$ a continuous nonvanishing map.
Define maps $\Psi \co X \times \Rplus \to \R^n$
and $\Psibar \co X \times \Rplus \to S^{n-1}$
by $\Psi(x,\lambda) = \lambda \varphi(x)$
and $\Psibar = \Psi / \| \Psi \|$.
Also define $\psi = \varphi / \| \varphi \| \co X \to S^{n-1}$.

Suppose that each point in $X \times \Rplus$ has a neighborhood $U$
such that the map $\Psi|_U \co U \to \R^n$ is convex
and such that the map $\Psibar|_U \co U \to \Psibar(U) \subset S^{n-1}$
is open to its image.

Then each point in $X$ has a neighborhood $V$
such that the map $\psi|_V$ is spherically convex
and is open to its image.
\end{Corollary}

\begin{Remark} \labell{cone}
Let $D_r \subset \R^2$ denote the closed disc of radius $r$
and center $(r,0)$.   Let 
$$ X = \{ (x,y,1+r) \in \R^3 \ | \ (x,y) \in D_r \text{ and } 0<r<1 \}.$$
Let $\varphi \co X \to \R^3$ be the inclusion map
and define $\Psi$ and $\psi$ as in Lemma~\ref{up-down}.
Then $\Psi$ is convex (by Lemma~\ref{slide'}),
but it is not open as a map to its image, and neither is $\psi$.
This shows that, in Lemma~\ref{up-down},
convexity of the map $\Psi|_U$ does not imply 
that the map $\psi|_V$ is open to its image.
Compare with Remark~\ref{projection not open}.
\end{Remark}

\begin{Lemma}\labell{to cone}
Let $C$ be a subset of $\R^n$ such that $\Rplus \cdot C = C$.
Let $U$ be a Hausdorff space and $\Psi \co U \to C$
a continuous nonvanishing open map.
Then $\Psibar := \Psi / \| \Psi \| \co U \to C \cap S^{n-1}$
is an open map.
\end{Lemma}

\begin{proof}
The map $\Psibar$ is the composition of three maps:
the map $\Psi$ from $U$ to $C \ssminus \{ 0 \}$,
the map $w \mapsto (\| w \| , \frac{w}{\| w \|})$
from $C \ssminus \{0\}$ to $\Rplus \times (C \cap S^{n-1})$,
and the projection map $\Rplus \times (C \cap S^{n-1}) \to C \cap S^{n-1}$.
The first of these maps is open by assumption;
the second is open because it is a homeomorphism;
the third is open by the definition of the product topology.
Being the composition of three open maps, $\Psibar$ is open.
\end{proof}

\begin{Proposition} \labell{stretch}
Let $X$ be a Hausdorff topological space and
$\varphi \co X \to \R^n$ a nonvanishing continuous map.
Define maps $\Psi \co X \times \Rplus \to \R^n$
and $\Psibar \co X \times \Rplus \to S^{n-1}$
by $\Psi(x,\lambda) = \lambda \varphi(x)$
and $\Psibar = \Psi / \| \Psi \| $.

Assume that each point in $X \times \Rplus$ has a neighborhood $U$
such that the map $\Psi|_U \co U \to \R^n$ is convex
and such that the map $\Psibar|_U \co U \to \Psibar(U) \subset S^{n-1}$
is open to its image.

Assume that $X$ is compact and connected;
assume that the image of the map $\varphi$ is not contained
in a two dimensional subspace of $\R^n$.

Then, for every two points $y_0$ and $y_1$ in $X \times \R_{>0}$,
if the segment $[\Psi(y_0),\Psi(y_1)]$ does not contain the origin,
there exists a path
$\gamma \co [0,1] \to X \times \R_{>0}$
such that $\gamma(0) = y_0$, $\gamma(1) = y_1$,
and $\Psi \circ \gamma$ is weakly monotone straight.
Also, the map $\Psi$ is open as a map to its image.
\end{Proposition}

\begin{proof}
Let
$$\psi = \varphi / \| \varphi \| \co X \to S^{n-1}.$$

By Corollary~\ref{every},
every point in $X$ has a neighborhood $V$
such that the map $\psi|_V$ is spherically convex and is open to its image.

Because the image of $\varphi$ is not contained
in a two dimensional subspace of $\R^n$,
the image of $\psi$ is not contained in a great circle.

Let $y_0 = (x_0,\lambda_0)$
and $y_1 = (x_1,\lambda_1)$ be two points in $X \times \Rplus$
such that the segment $[\Psi(y_0),\Psi(y_1)]$
does not contain the origin.
Then $x_0$ and $x_1$ are points in $X$
such that $\psi(x_1) \neq - \psi(x_0)$.

By Corollary~\ref{shortening},
there exists a path $x(t)$, $0 \leq t \leq 1$,
such that $x(0) = x_0$ and $x(1) = x_1$,
and such that $\psi(x(t))$ is a weakly monotone short geodesic.

By Lemma~\ref{straight2} there exists a weakly monotone straight path
$$ \gammabar \co [0,1] \to \R^n $$
such that $\gammabar(0) = \lambda_0 \varphi(x_0)$,
$\gammabar(1) = \lambda_1 \varphi(x_1)$,
and $\gammabar(t) / \| \gammabar(t) \| = \psi(x(t))$.
Then $\gammabar(t) = \lambda(t) \varphi(x(t))$,
where $\lambda(t) = \| \gammabar(t) \| / \| \varphi(x(t))\|$.
The path
$ \gamma(t) := (x(t),\lambda(t))$ in $X \times \Rplus$
satisfies $\gamma(0) = y_0$, $\gamma(1) = y_1$,
and $\Psi(\gamma(\cdot)) = \gammabar(\cdot)$ is weakly monotone straight.

By Corollary~\ref{shortening},
the map $\psi \co X \to S^{n-1}$ is open as a map to its image.
From this it follows that the map
$$ X \times \Rplus \, \to \, S^{n-1} \times \Rplus \quad , \quad
(x,\lambda) \mapsto (\psi(x),\lambda) $$
is open as a map to its image.
From the commuting diagram
\begin{equation} \labell{diagram}
 \xymatrix{
 & X \times \Rplus \ar[rr]^{ (x,\lambda) \mapsto (\psi(x),\lambda) }
 \ar[d]_{ (x,\lambda) \mapsto (x, \frac{\lambda}{\| \varphi(x) \| } ) }
 & & S^{n-1} \times \Rplus \ar[d]^{ (\alpha,\lambda) \mapsto \lambda\alpha} \\
 & X \times \Rplus \ar[rr]^{\Psi} & & \R^n \ssminus \{ 0 \} ,
}
\end{equation}
in which the vertical arrows are homeomorphisms
and the top arrow is open to its image,
it follows that $\Psi$ is open to its image.
\end{proof}

\begin{Example} \labell{Rn minus origin}
Let $\varphi \co S^{n-1} \to \R^n$ be the inclusion map.
Then the map
$$ \Psi \co S^{n-1} \times \Rplus \to \R^n
 \quad , \quad (x,\lambda) \mapsto \lambda x, $$
satisfies the assumptions and the conclusion of Proposition~\ref{stretch}.
Note that the map $\Psi$ itself is not convex.
\end{Example}

\begin{Example} \labell{dim2}
Let $\varphi \co [-\pi,\pi] \to \R^2$ be the map
$$\varphi(t) = ( \cos t , \sin t ) .$$
Then the map
$$ \Psi \co [-\pi,\pi] \times \Rplus \to \R^2
\quad , \quad (t,\lambda) \mapsto ( \lambda \cos t , \lambda \sin t) $$
is not open as a map to its image, $\R^2 \ssminus \{ 0 \}$,
although every point has a neighborhood on which
the map is convex and is open as a map to its image.
Also, if $[t_0,t_1] \subset [-\pi,\pi]$ is a subinterval
of length $> \pi$, then, for any $\lambda_0,\lambda_1 \in \Rplus$,
the segment $[\Psi(t_0,\lambda_0),\Psi(t_1,\lambda_1)]$
does not contain the origin, but the points $(t_0,\lambda_0)$
and $(t_1,\lambda_1)$ cannot be connected by a path
in $[-\pi,\pi] \times \Rplus$ whose image under $\Psi$ is this segment.
Thus, the conclusions of Proposition~\ref{stretch}
do not always hold if we allow the image of $\varphi$
to be contained in a two dimensional space.
\end{Example}

\section{Excision of a neighborhood of zero}
\labell{sec:excision}

\begin{Remark} \labell{restrict open}
We repeatedly use the following properties 
of a continuous map that is open as a map to its image: 
\begin{enumerate}
\item the restriction of this map
to an open subset is open as a map to its image, and 
\item the restriction of this map to the preimage of any set is also open to its image.
\end{enumerate} 
Thus, given a map $\varphi \co U \to \R^n$,
if $\varphi$ is open as a map to its image, then,
for any open subset $V \subset U$, 
the map $\varphi|_{\varphi\Inv(\del B_\delta)\cap V}$ is open 
as a map to its image.
\end{Remark}

\begin{Lemma} \labell{disjoint'}
Let $X$ be a Hausdorff topological space
and $\varphi \co X \to \R^n$ a continuous proper map.
Suppose that for every point $x$ of $\varphi\Inv(0)$
there exists an open neighborhood $U_x$ of $x$ in $X$
and a closed convex cone $C_x$ in $\R^n$ with vertex at the origin
such that
\begin{itemize}
\item
The cone $C_x$ is not contained in a two dimensional subspace of $\R^n$.
\item
The image $\varphi(U_x)$ is an open subset of $C_x$.
\item
The map $\varphi|_{U_x} \co U_x \to \varphi(U_x)$
is open and convex and, for sufficiently small $\delta > 0$, every point in $\varphi\Inv(\del B_\delta) \cap U_x$ has a neighborhood $V \subset U_x$ such that
the restriction $\varphi|_{\varphi\Inv(\del B_\delta) \cap V}$
is spherically convex.
\end{itemize}

Then there exist open subsets $W_1,\ldots,W_N$ of $X$
and closed convex cones $C_1,\ldots,C_N$
with vertex at the origin
and there exists $\eps > 0$ such that
\begin{itemize}
\item
The sets $W_1,\ldots,W_N$ are disjoint and their union
is equal to $\varphi\Inv(B_\eps)$.
\item
For each $i$, the image $\varphi(W_i)$ is equal to $C_i \cap B_\eps$.
\item
For each $i$, the map $\varphi|_{W_i} \co W_i \to \varphi(W_i)$
is open and convex.
\item
For each $i$ and each $0 < \delta < \eps$,
the restriction
$\varphi|_{ \varphi\Inv(\del B_\delta) \cap W_i }$
is spherically convex,
and its image is not equal to a pair of antipodal points.
\end{itemize}
\end{Lemma}

\begin{proof}
Because the level set $\varphi\Inv(0)$ is compact and locally
connected, it has finitely many connected components,
$[x_1],\ldots,[x_N]$.

For each $i=1,\ldots,N$, by applying Lemma~\ref{good U}
to the set $K = [x_i]$, we choose an open subset $U_i$ of $X$
and a closed convex cone $C_i$ with vertex at the origin such that
\begin{itemize}
\item[(a)]
The set $U_i$ contains the component $[x_i]$ of $\varphi\Inv(0)$.
\item[(b)]
The cone $C_i$ is not contained in a two dimensional subspace of $\R^n$.
\item[(c)]
The image $\varphi(U_i)$ is an open subset of $C_i$.
\item[(d)]
The map $\varphi|_{U_i} \co U_i \to \varphi(U_i)$ is open and convex,
and, for sufficiently small $\delta > 0$,
every point in $\varphi\Inv(\del B_\delta) \cap U_i$
has a neighborhood $U \subset U_i$ such that
$\varphi|_{\varphi\Inv(\del B_\delta) \cap U}$ is spherically convex.
\end{itemize}

We proceed in analogy with the proofs of Lemma~\ref{spherical intgoalB}
and of~\cite[Proposition~17]{karshon-marshall}.

Choose disjoint open subsets $\calO_1, \ldots, \calO_N$ of $X$
such that $\calO_i$ contains $[x_i]$ and is contained in $U_i$;
this is possible because $\varphi\Inv(0)$ is compact
and $X$ is Hausdorff.

By properties (c) and (d), there exist $\eps_i > 0$
such that $\varphi(U_i) \cap B_{\eps_i} = C_i \cap B_{\eps_i}$
and such that, if $0 < \delta < \eps_i$,
every point in $\varphi\Inv(\del B_\delta) \cap U_i$
has a neighborhood $U \subset U_i$
such that $\varphi|_{\varphi\Inv(\del B_\delta) \cap U}$
is spherically convex.

Choose $\eps>0$ smaller than $\eps_1,\ldots,\eps_N$
and such that the preimage $\varphi\Inv(B_\eps)$
is contained in $\calO_1 \cup \ldots \cup \calO_N$;
this is possible because $\varphi$ is proper.
Then every connected set that meets $[x_i]$ and is contained
in $\varphi\Inv(B_\eps)$ must be contained in $\calO_i$.

Set $W_i = U_i \cap \varphi\Inv(B_{\eps})$.
Because $\varphi|_{U_i}$ is convex, so is $\varphi|_{W_i}$;
in particular, $W_i$ is connected.
Because $W_i$ meets $[x_i]$ and is contained in $\varphi\Inv(B_\eps)$
it must be contained in $\calO_i$.
So $W_1, \dots, W_N$ are disjoint, their union is $\varphi\Inv(B_{\eps})$,
and $\varphi|_{W_i} \co W_i \to C_i$
is convex and open.

Fix $\delta$ such that $0 < \delta < \eps$.
The sets $W_i \cap \varphi\Inv(\del B_\delta)$
are open in $\varphi\Inv(\del B_\delta)$, disjoint, 
and they cover $\varphi\Inv(\del B_\delta)$.
So each of these sets is closed in $\varphi\Inv(\del B_\delta)$,
hence compact.
Let $Y$ be a connected component of $\varphi\Inv(\del B_\delta) \cap W_i$.
By property (d), every point in $Y$ has a neighborhood $U$
such that $\varphi|_{Y \cap U}$ is spherically convex; 
it is also open as a map to its image, because $\varphi|_{W_i}$ is.
By properties (b) and (c), the image $\varphi(Y)$
is not contained in a great circle of $\del B_\delta$.
By Corollary~\ref{shortening}, it follows that
the map $\varphi|_Y$ itself is spherically convex.

Because $\varphi(Y)$ is open and closed in $\del B_\delta \cap C_i$,
it is equal to $\del B_\delta \cap C_i$.
Because $\varphi|_{W_i}$ has connected level sets,
it follows that the connected component $Y$ is equal to the entire
space $\varphi\Inv(\del B_\delta) \cap W_i$.
Thus, the map $\varphi|_{ \varphi\Inv(\del B_\delta) \cap W_i }$
is spherically convex
and its image is not equal to a pair of antipodal points.
\end{proof}

\begin{Lemma} \labell{X'connected}
Let $X$ be a Hausdorff topological space,
$n \geq 2$, 
and $\varphi \co X \to \R^n$ a continuous map.
Let $\eps > 0$.
Suppose that there exist open subsets $W_1,\ldots,W_N$ of $X$
such that
\begin{itemize}
\item
The sets $W_1,\ldots,W_N$ are disjoint and their union contains
$\varphi\Inv(\ol{B}_\eps)$.
\item
For each $i$, the restriction $\varphi|_{ \varphi\Inv(\del B_\eps) \cap W_i }$
is a spherically convex map,
and its image is not equal to a pair of antipodal points.
\end{itemize}
Suppose that $X$ is path connected.
Then the excised space
$X' := X \ssminus \varphi\Inv(B_\eps)$
is also path connected.
\end{Lemma}

\begin{proof}
Denote $W_0 := X \ssminus \varphi\Inv(\ol{B}_\eps)$.
Then $W_0,W_1,\ldots,W_N$ is an open covering of $X$.

Let $x_0$ and $x_1$ be two points in $X'$.
Let $\gamma \co [0,1] \to X$ be a path such that
$\gamma(0) = x_0$ and $\gamma(1) = x_1$.
Every $t \in [0,1]$ has a neighborhood $J$ in $[0,1]$
such that $\gamma(J)$ is entirely contained in one of the sets
$W_0,W_1,\ldots,W_N$; we may assume that $J$ is an interval.
Because the interval $[0,1]$ is compact, there exists a partition
$0 = t_0 < t_1 < \ldots < t_M = 1$,
and, for each $1 \leq j \leq M$, 
an integer $i_j \in \{0, \dots, N\}$
such that
the image $\gamma([t_{j-1},t_j])$ is contained in the set $W_{i_j}$.

After possibly passing to a coarser partition of $[0,1]$,
we assume that no two consecutive sets
in the sequence $W_{i_1}, \ldots, W_{i_M}$ are equal.
Also, because any two consecutive sets in this sequence
meet at a division point $\gamma(t_j)$
whereas the sets $W_1, \ldots, W_N$ are disjoint,
of any two consecutive sets in the sequence at least one must be $W_0$.
So, because $W_0$ is contained in $X'$,
the interior division points $\gamma(t_1),\ldots,\gamma(t_{M-1})$
must all be in $X'$.
By assumption, the endpoints $\gamma(t_0) = x_0$ and $\gamma(t_M) = x_1$
are also in $X'$.

We now concentrate on the $j$th subinterval, $[t_{j-1},t_j]$.
If $\gamma(t) \in X'$ for all $t \in [t_{j-1},t_j]$,
then we define $\gamma_j \co [t_{j-1},t_j] \to X' $
to be the restriction $\gamma|_{[t_{j-1},t_j]}$.

Otherwise, let $a$ and $b$ to be the infimum and supremum of the set 
$\{t \in [t_{j-1},t_j] \mid \gamma(t) \not\in X' \}$.
So $[a,b] \subset [t_{j-1},t_j]$,
$\gamma(a) \in \varphi\Inv(\del B_\eps)$,
$\gamma(b) \in \varphi\Inv(\del B_\eps)$,
and $\gamma(t) \in X'$
for $t_{j-1} \leq t \leq a$ and for $b \leq t \leq t_j$.
Because the image $\gamma([t_{j-1},t_j])$ is not contained in $X'$,
it must be contained in one of the sets $W_1,\ldots W_N$,
say, in $W_i$.
Because the restriction $\varphi|_{\varphi\Inv(\del B_\eps) \cap W_i}$
is spherically convex
and its image is not equal to a pair of antipodal points,
the set $\varphi\Inv(\del B_\eps) \cap W_i$ is path connected.
So there exists a path $\tgamma \co [a,b] \to \varphi\Inv(\del B_\eps)$
connecting $\gamma(a)$ and $\gamma(b)$.
Define
$$ \gamma_j(t) = \begin{cases}
 \gamma(t) & t_{j-1} \leq t \leq a \\
 \tgamma(t) & a \leq t \leq b \\
 \gamma(t) & b \leq t \leq t_j .
\end{cases} $$

The concatenation of the paths $\gamma_1, \ldots, \gamma_M$
lies entirely in $X'$ and connects $x_0$ to $x_1$.
\end{proof}

\begin{Lemma} \labell{localopen'}
Let $W$ be a Hausdorff topological space.
Let $C$ be a subset of $\R^n$ such that $\Rplus \cdot C = C$.
Let $\varphi \co W \to C$ be a continuous open map.
Define maps
$ \Psi \co W \times \Rplus \to C $
and
$ \Psibar \co (W\ssminus \varphi\Inv(0)) \times \Rplus \to S^{n-1} \cap C $
by
$ \Psi(x,\lambda) = \lambda \varphi(x) $
and
$ \Psibar = \Psi / \| \Psi \| $.

Let $0 < \eps < \eps'$ be positive numbers.
Suppose that the image of $\varphi$ contains $B_{\eps'} \cap C$.
Let
$$ W' = W \ssminus \varphi\Inv(B_\eps).$$
Then
\begin{itemize}
\item
The map $\Psi|_{W' \times \Rplus} \co W' \times \Rplus \to C$
is open and its image is $C \ssminus \{ 0 \}$.
\item
The map $\Psibar|_{W' \times \Rplus} \co W' \times \Rplus
\to S^{n-1} \cap C$ is open.
\end{itemize}
\end{Lemma}

Note that, above, $C$ is not necessarily closed.

\begin{proof}
We begin with three consequences of the condition $\Rplus \cdot C = C$.

First, we identify the image of the map $\Psi|_{W' \times \Rplus}$.
This image is equal to
 $\Rplus \cdot \left( \varphi(W) \ssminus B_\eps \right)$,
which is contained in the set $\Rplus \cdot \left( C \ssminus B_\eps \right)$
and contains the set
 $\Rplus \cdot \left( C \cap ( B_{\eps'} \ssminus B_\eps ) \right)$;
both of these sets are equal to $C \ssminus \{ 0 \}$.
So the image of $\Psi|_{W' \times \Rplus}$ is $C \ssminus \{ 0 \}$ as claimed.

Next, we note that if $I$ is an open interval contained in $\Rplus$
and $\calO$ is open in $C$, then the sets
$I \cdot \calO$ and $I \cdot (\calO \cap \del B_\eps)$
are open in $C$.

Finally, suppose that
$\Psi \co W' \times \Rplus \to C $ is open.
Then $\Psi$ is also open as a map to $C \ssminus \{ 0 \}$.
Since the central projection $\pi\co \R^n \ssminus \{0\} \to S^{n-1}$ 
given by $\pi(y)=y/\|y\|$ is open,
and since $\pi\Inv(S^{n-1} \cap C) = C \ssminus \{0\}$,
it follows that 
$\Psibar = \pi \circ \Psi \co W' \times \Rplus \to S^{n-1} \cap C$ 
is also open, as a composition of two open maps.

It remains to show that the map $\Psi \co W' \times \Rplus \to C$ is open.

The sets of the form $Y' \times I$, where $I$ is an open interval contained in $\Rplus$,
and where $Y$ is an open subset of $W$ and
$Y' = Y \ssminus \varphi\Inv(B_\eps)$,
form a basis to the topology of $W' \times \Rplus$.
So we need to show that for every such set the image
$$ \Psi(Y' \times I) = I \cdot \left( \varphi(Y) \ssminus B_\eps \right) $$
is open in $C$.
Let $\mu$ be a point in $I \cdot (\varphi(Y) \ssminus B_\eps )$,
say, $\mu = \lambda \cdot \varphi(y')$
with $y' \in Y \ssminus \varphi\Inv(B_\eps)$
and $\lambda \in I$.

If $y'$ is actually in $Y \ssminus \varphi\Inv(\ol{B}_\eps)$,
then $\varphi(Y) \ssminus \ol{B}_\eps$ is an open neighborhood
of $\varphi(y')$ in $C$, because $\varphi(Y)$ is open in $C$.
It follows that
$I \cdot \left( \varphi(Y) \ssminus \ol{B}_\eps \right)$
is an open neighborhood of $\lambda \cdot \varphi(y')$ in $C$.

Now suppose that $y' \in \varphi\Inv(\del B_\eps)$.
Because $\varphi(Y)$ is open in $C$,
its intersection with $\del B_\eps$
is open in $C \cap \del B_\eps$.
It follows that the set $I \cdot \left( \varphi(Y) \cap \del B_\eps \right)$
is an open neighborhood of $\lambda \cdot \varphi(y')$ in $C$.

In either case, we found an open neighborhood of $\mu$ in $C$
that is contained in $I \cdot \left( \varphi(Y) \ssminus B_\eps \right)$.
This completes the proof of the lemma.
\end{proof}

\begin{Lemma} \labell{localconvex'}
Let $W$ be a Hausdorff topological space,
let $C$ be a closed convex cone in $\R^n$ with vertex at the origin,
and let $\varphi \co W \to C$ be a continuous map.
Define maps
$ \Psi \co W \times \Rplus \to C  $
and $ \Psibar \co
      (W\ssminus \varphi\Inv(0)) \times \Rplus \to S^{n-1} \cap C $
by
$ \Psi(x,\lambda) = \lambda \varphi(x) $
and
$ \Psibar = \Psi / \| \Psi \| $.
Let $0 < \eps < \eps'$ be positive numbers.
Let
$$ W' = W \ssminus \varphi\Inv(B_\eps).$$

Assume that
\begin{itemize}
\item
$\varphi(W) = B_{\eps'} \cap C$.
\item
The map $\varphi$ is convex and open to its image.
\item
The map $\varphi|_{\varphi\Inv(\del B_\eps)}$ is spherically convex.
\end{itemize}

Then every point in $W' \times \Rplus$
has a neighborhood $U'$ in $W' \times \Rplus$
such that the map $\Psi|_{U'}$ is convex
and the map $\Psibar|_{U'}$ is open to its image.
\end{Lemma}

\begin{proof}
Let $H$ be an open half--space in $\R^n$ whose boundary contains the origin,
and let
$$ U' = \left(
   \varphi\Inv(H) \ssminus \varphi\Inv(B_\eps) \right) \times \Rplus .$$
We will prove that the map $\Psi|_{U'}$ is convex
and that the map $\Psibar|_{U'}$ is open to its image.
This will be enough, because the sets $U'$, for different choices
of $H$, form an open covering of $W' \times \Rplus$.

By Lemma~\ref{localopen'}
applied to $\varphi\co \varphi\Inv(H) \to H \cap C$,
the map $\Psibar|_{U'}$ is open as a map to its image.
It remains to show that the map $\Psi|_{U'}$ is convex.
This will overlap the proof of Lemma~\ref{X'connected},
in that we will connect two points of $W'$ by a path
that lies entirely in $W'$, but here we must take care
to obtain a path whose composition with $\varphi$
can be ``straightened" by multiplication by a positive function.

Let
$$ U = \varphi\Inv(H) \times \Rplus .$$
Because $\varphi$ is a convex map and $H$ is a convex set,
the restriction $\varphi|_{\varphi\Inv(H)}$ is a convex map.
It follows by Lemma~\ref{slide'} that $\Psi|_U$ is a convex map.

Let $y_0 = (x_0, \lambda_0)$ and $y_1 = (x_1, \lambda_1)$
be two points in $U'$.
The segment $[\Psi(y_0), \Psi(y_1)]$
is contained in $H$, so it does not pass through the origin.
Because $\Psi|_U$ is convex,
there exists a path $y(t) = (x(t), \lambda(t))$ in $U$
connecting $y_0$ and $y_1$ with $\Psi \circ y$ weakly monotone straight.
By Lemma~\ref{straight}, the path $\Psibar(y(\cdot))$
is a weakly monotone short geodesic.

Let
$$ \psi := \varphi / \| \varphi \|
   \co W \ssminus \varphi\Inv(0)  \to S^{n-1} .$$
Then
$$\Psibar (x,\lambda) = \psi(x),$$
and in particular $\Psibar(y(t)) = \psi(x(t))$.
So $\psi(x(\cdot))$ is a weakly monotone short geodesic.

Since $x_0$ and $x_1$ are in $W'$,
the path $\varphi(x(\cdot))$ starts and ends outside $B_\eps$.
If the path $\varphi(x(\cdot))$ happens to lie entirely outside $B_\eps$,
then the path $y(\cdot)$ lies entirely in $U'$, and we are done.
Otherwise, let $[a,b] \subset [0,1]$ be such that
$\varphi(x(a)) \in \del B_\eps$, $\varphi(x(b)) \in \del B_\eps$,
and $\varphi(x(t)) \not\in B_\eps$
for $0 \leq t \leq a$ and for $b \leq t \leq 1$.
For instance, we may take $a$ and $b$ to be the infimum and supremum
of the set of times $t$ at which $\varphi(x(t))$
is in the ball $B_\eps$.

Because
$\varphi|_{\varphi\Inv(\del B_\eps \cap H)}$
is spherically convex (cf.~Remark~\ref{restrict})
and its image is not equal to a pair of antipodal points,
there exists a path $\tilde{x}(t)$, for $a \leq t \leq b$,
connecting $x(a)$ and $x(b)$
and lying in $\varphi\Inv(\del B_\eps \cap H)$,
such that $\varphi(\tilde{x}(t))$ is a weakly monotone
short geodesic.
Necessarily, $\tilde{x}(t)$ is in $W'$.
We define $\tilde{x}(t)$ to be equal to $x(t)$
on the segments $[0,a]$ and $[b,1]$.
Then $\psi(\tilde{x}(t))$ is a path in $S^{n-1}$
whose restriction to each of the segments $[0,a]$, $[a,b]$, and $[b,1]$
is a weakly monotone short geodesic.
But the value of this path at the points $0$, $a$, $b$, $1$
coincide with the values of $\psi(x(t))$, which is a
weakly monotone short geodesic on the entire segment $[0,1]$.
It follows that $\psi(\tilde{x}(t))$ is also a weakly monotone
short geodesic on the entire segment $[0,1]$.

Let $\lambda_0' = \| \Psi(y_0) \|$
and $\lambda_1' = \| \Psi(y_1) \|$.
Then $\Psi(y_0) = \lambda_0' \psi(\tilde{x}(0))$
and $\Psi(y_1) = \lambda_1' \psi(\tilde{x}(1))$.
Lemma~\ref{straight2} implies that there exists
a weakly monotone straight path $\gammabar \co [0,1] \to \R^n$
from $\Psi(y_0)$ to $\Psi(y_1)$
such that $\gammabar/ \| \gammabar \| = \psi (\tilde{x}(\cdot))$.
In particular, $\gammabar(t)$ is a positive multiple
of $\varphi(\tilde{x}(t))$.
So there exists a continuous function $\tlambda \co [0,1] \to \Rplus$
such that $\gammabar(t) = \tlambda(t) \varphi(\tilde{x}(t))$;
namely,
$\tlambda(t) = \| \gammabar(t) \| / \| \varphi(\tilde{x}(t) ) \| $.
So $\tilde{y}(\cdot) = \bigl( \tilde{x}(\cdot) , \tlambda(\cdot) \bigr)$
is a path in $U'$ from $y_0$ to $y_1$ such that $\Psi \circ y$
is weakly monotone straight.
\end{proof}

\begin{Lemma} \labell{A}
Let $X$ be a compact connected Hausdorff topological space
and $\varphi \co X \to \R^n$ a continuous map.
Define maps $\Psi \co X \times \Rplus \to \R^n$
and $\Psibar \co (X \ssminus \varphi\Inv(0)) \times \Rplus \to S^{n-1}$
by $\Psi(x,\lambda) = \lambda \varphi(x)$
and $\Psibar = \Psi / \| \Psi \|$.
Assume that
\begin{itemize}
\item
The image $\varphi(X)$ contains $0$.
\item
For every point $x$ of $X$ with $\varphi(x) = 0$
there exists an open neighborhood $U_x$ of $x$ in $X$
and a closed convex cone $C_x$ in $\R^n$
with vertex at the origin such that
\begin{itemize}
\item[\dash]
The cone $C_x$ is not contained in a two dimensional subspace of $\R^n$.
\item[\dash]
The image $\varphi(U_x)$ is an open subset of $C_x$.
\item[\dash]
The map $\varphi|_{U_x} \co U_x \to \varphi(U_x)$ is convex and open,
and, for sufficiently small $\delta>0$, 
every point in $\varphi\Inv(\del B_\delta) \cap U_x$ 
has a neighborhood $V\subset U_x$ such that the restriction
$\varphi|_{\varphi\Inv(\del B_\delta) \cap V}$ 
is spherically convex.
\end{itemize}
\item
For every point $(x,\lambda)$ of $X \times \Rplus$ with $\varphi(x) \neq 0$,
every neighborhood of $(x,\lambda)$ in $X \times \Rplus$
contains a smaller neighborhood, $U$,
such that the map $\Psi|_U$ is convex
and such that the map $\Psibar|_U$ is open as a map to its image.
\end{itemize}
Then, for every sufficiently small positive number $\eps$,
the following results hold.  Let
\begin{equation} \labell{X'}
 X' = X \ssminus \varphi\Inv(B_\eps) .
\end{equation}
Then
\begin{enumerate}
\item
For every two points $y_0$ and $y_1$ in $X' \times \Rplus$,
if the segment $[\Psi(y_0),\Psi(y_1)]$
does not contain the origin,
then there exists a path $\gamma \co [0,1] \to X' \times \Rplus$
such that $\gamma(0) = y_0$, $\gamma(1) = y_1$,
and $\Psi \circ \gamma$ is weakly monotone straight.
\item
The map $\Psi \co X' \times \Rplus \to \Psi(X' \times \Rplus)$
is open as a map to its image.
\item
The image $\Psi(X' \times \Rplus)$ is equal
to $\Psi(X \times \Rplus) \ssminus \{ 0 \}$.
\end{enumerate}
\end{Lemma}

\begin{proof}
The space $X$ and the map $\varphi$ satisfy the assumptions
of Lemma~\ref{disjoint'}.  
Let $W_1,\ldots,W_N$ be open subsets of $X$,
let $C_1,\ldots,C_N$ be closed convex cones with vertex at the origin,
and let $\eps'$ be a positive number, such that
\begin{itemize}
\item
The sets $W_1,\ldots,W_N$ are disjoint and their union
is equal to $\varphi\Inv(B_{\eps'})$.
\item
For each $i$, the image $\varphi(W_i)$ is equal to $C_i \cap B_{\eps'}$.
\item
For each $i$, the map $\varphi|_{W_i} \co W_i \to \varphi(W_i)$
is open and convex, and, for each $0 < \eps < \eps'$,
the restriction $ \varphi|_{\varphi\Inv(\del B_\eps) \cap W_i } $
is spherically convex
and its image is not equal to a pair of antipodal points. 
\end{itemize}

Let $\eps$ be any positive number such that $0 < \eps < \eps'$,
and let $X' = X \ssminus \varphi\Inv(B_\eps)$.

Because $X$ is connected and locally path connected, $X$ is path connected.
The space $X$, the map $\varphi$, and the number $\eps$ satisfy
the assumptions of Lemma~\ref{X'connected}.
Thus, the excised space $X'$ is path connected.

For each $1 \leq i \leq N$, the set $W_i$, the cone $C_i$,
the map $\varphi|_{W_i} \co W_i \to C_i$,
and the numbers $\eps'$ and $\eps$
satisfy the assumptions of Lemma~\ref{localopen'}.
By the first part of that lemma,
$$ \Psi ((W_i \cap X') \times \Rplus) 
   = \Psi (W_i \times \Rplus) \ssminus \{ 0 \} .  $$
Because $ X = X' \cup \bigcup\limits_{i=1}^N W_i$,
this implies that
\begin{equation} \labell{star}
   \Psi (X' \times \Rplus) = \Psi( X \times \Rplus ) \ssminus \{ 0 \} , 
\end{equation}
which is item (3) that we had set to prove.

By the assumptions on $\varphi\Inv(0)$, the image $\varphi(X)$ 
contains a subset (namely, $\varphi(U_x)$ for $\varphi(x)=0$) 
that is not contained in any two dimensional subspace of $\R^n$. 
Rewriting~\eqref{star} 
as $\Rplus \cdot \varphi(X') = \Rplus \cdot \varphi(X) \ssminus \{ 0 \}$,
we deduce that $\varphi(X')$ is not contained
in a two dimensional subspace of $\R^n$ either.

Let $(x,\lambda)$ be a point of $X' \times \Rplus$.
If $x$ belongs to one of the sets $W_1,\ldots,W_N$,
then Lemma~\ref{localconvex'}
guarantees the existence of a neighborhood $U'$
of $(x,\lambda)$ in $X' \times \Rplus$
such that the map $\Psi|_{U'}$ is convex
and the map $\Psibar|_{U'}$ is open to its image.
If $x$ does not belong to any $W_i$, then 
$\varphi(x)$ is outside $B_{\eps'}$,
and $(X \ssminus \varphi\Inv(\ol{B}_\eps)) \times \Rplus$
is a neighborhood of $(x,\lambda)$ in $X \times \Rplus$;
by assumption, it contains a smaller neighborhood, $U$,
such that $\Psi|_U$ is convex and $\Psibar|_U$ is open to its image.

We have just shown that every point in $X' \times \Rplus$
has a neighborhood $U'$ in $X' \times \Rplus$
such that the map $\Psi|_{U'}$ is convex
and the map $\Psibar|_{U'}$ is open to its image.
We have also shown that $X'$ is connected
and that $\varphi(X')$ is not contained in a two dimensional subspace
of $\R^n$.  Also, being a closed subset of a compact space,
$X'$ is compact.
Thus, the space $X'$ and the map $\varphi|_{X'}$ satisfy 
the assumptions of Proposition~\ref{stretch}.
The conclusions of this proposition are exactly the items (1) and (2)
that we had set to prove.
\end{proof}

\begin{Definition}\labell{path lifting}
A continuous map $\varphi\co X \to B$ 
has the \emph{weak path lifting property}
if, for any point $x$ in $X$ 
and path $\gamma \co [0,1] \to B$
with $\gamma(0) = \varphi(x)$,
there exists a path $\tgamma \co [0,1] \to X$
and a weakly monotone reparametrization
$s \co [0,1] \to [0,1]$, with $s(0)=0$ and $s(1)=1$,
such that $\tgamma(0) = x$ and $\varphi(\tgamma(t)) = \gamma(s(t))$
for all $t \in [0,1]$.
\end{Definition}

\begin{Lemma} \labell{usage of path lifting}
Let $U$ be a Hausdorff topological space.
Let $C$ be a closed convex cone in $\R^n$
with vertex at the origin.
Let $\eps > 0$.
Let $\varphi \co U \to C$ be a continuous map.

Assume that
\begin{itemize}
\item
The image $\varphi(U)$ contains $B_{\eps} \cap C$.
\item
The map $\varphi \co U \to \varphi(U)$
has the weak path lifting property,
and its level sets are path connected.
\end{itemize}

Then
\begin{itemize}
\item
The map
$\varphi|_{\varphi\Inv(B_\eps)}$ is convex.
\item
For every $0 < \delta < \eps$,
every point $x$ in $\varphi\Inv(\del B_\delta)$
has a neighborhood $V_x$
such that the map $\varphi|_{\varphi\Inv(\del B_\delta) \cap V_x}$
is spherically convex.
\end{itemize}
\end{Lemma}

\begin{proof}
For any subset $E$ of $\R^n$, the restriction
$\varphi|_{\varphi\Inv(E)} \co \varphi\Inv(E) \to E \cap \varphi(U)$
has the weak path lifting property,
and its level sets are path connected.
This follows from the analogous properties of $\varphi$.
If the set $E \cap \varphi(U)$ is convex,
it follows that the map $\varphi|_{\varphi\Inv(E)}$ is convex.
If the set $E \cap \varphi(U)$ is contained in $\del B_\delta$
and is spherically convex, it follows that the map
$\varphi|_{\varphi\Inv(E)} \co \varphi\Inv(E) \to \del B_\delta$
is spherically convex.

The first part of the lemma follows by setting $E = B_\eps$
and noting that the set $E \cap \varphi(U) = B_\eps \cap C$ is convex.

For the second part of the lemma, fix $0 < \delta < \eps$
and fix $x$ in $\varphi\Inv(\del B_\delta)$.
Let $H_x \subset \R^n$ be an open half--space whose boundary contains the origin
such that $\varphi(x) \in H_x$.
Let $E = H_x \cap \del B_\delta$.
Then $E \cap \varphi(U) = \del B_\delta \cap H_x \cap C$
is a spherically convex subset of $\del B_\delta$,
and $\varphi\Inv(E) = \varphi\Inv(\del B_\delta) \cap \varphi\Inv(H_x)$.
So $V_x := \varphi\Inv(H_x)$
is a neighborhood of $x$ such that the map $\varphi|_{\varphi\Inv(\del B_\delta)\cap V_x}$ is spherically convex. 
\end{proof}

\begin{Proposition} \labell{with zero'}
Let $X$ be a compact connected Hausdorff topological space
and $\varphi \co X \to \R^n$ a continuous map.
Define maps $\Psi \co X \times \Rplus \to \R^n$
and $\Psibar \co (X \ssminus \varphi\Inv(0)) \times \Rplus \to S^{n-1}$
by $\Psi(x,\lambda) = \lambda \varphi(x)$
and $\Psibar = \Psi / \| \Psi \|$.
Assume that
\begin{itemize}
\item The image $\varphi(X)$ contains $0$. 	
\item
For every point $x$ of $X$ with $\varphi(x) = 0$,
there exists an open neighborhood $U_x$ of $x$ in $X$
and a closed convex cone $C_x$ in $\R^n$
with vertex at the origin such that
\begin{itemize}
\item[\dash]
The cone $C_x$ is not contained in a two dimensional subspace of $\R^n$.
\item[\dash]
The image $\varphi(U_x)$ is an open subset of $C_x$.
\item[\dash]
The map $\varphi|_{U_x} \co U_x \to \varphi(U_x)$ is open,
has the weak path lifting property,
and its level sets are path connected.
\end{itemize}
\item
For every point $(x,\lambda)$ of $X \times \Rplus$
with $\varphi(x) \neq 0$,
every neighborhood of $(x,\lambda)$ in $X \times \Rplus$
contains a smaller neighborhood $U$
such that the map $\Psi|_U$ is convex
and such that the map $\Psibar|_U$ is (defined and) open 
as a map to its image.
\end{itemize}
Then the following results hold.  Let
$$ C = \Psi (X \times \Rplus) = \Rplus \cdot \varphi(X).$$
\begin{enumerate}
\item
For any two points $y_0$ and $y_1$ in $X \times \Rplus$,
if $\Psi(y_0)$ and $\Psi(y_1)$ are not both zero,
there exists a path $\gamma \co [0,1] \to X \times \Rplus$
such that $\gamma(0) = y_0$ and $\gamma(1) = y_1$
and such that $\Psi \circ \gamma \co [0,1] \to \R^n$
is weakly monotone straight.
\item
For every $x \in \varphi\Inv(0)$, the cone $C_x$ is equal to $C$.
\item
The map $\Psi$ is open as a map to $C$.
\end{enumerate}
\end{Proposition}

\begin{proof}
For $x \in \varphi\Inv(0)$, that fact that the image $\varphi(U_x)$ 
is open in $C_x$ implies that $\varphi(U_x)$ 
contains $B_{\eps} \cap C_x$ for some $\eps >0$.
Applying Lemma~\ref{usage of path lifting} to $\varphi\co U_x \to C_x$,
and after replacing $U_x$ by $\varphi\Inv(B_\eps) \cap U_x$
(and continuing to denote it by $U_x$),
we get the following facts.
\begin{itemize}
\item[\dash]
$\varphi|_{U_x}\co U_x \to C_x$ is convex and open.
\item[\dash]
For every sufficiently small $\delta>0$,
every point in $\varphi\Inv(\del B_\delta)\cap U_x$ 
has a neighborhood $V \subset U_x$ such that 
$\varphi|_{\varphi\Inv(\del B_\delta)\cap V}$ is spherically convex. 
\end{itemize}

The space $X$ and the map $\varphi$ now satisfy the assumptions 
of Lemma~\ref{A}.
Let $\eps'$ be a positive number that is sufficiently small
so that the conclusions of Lemma \ref{A} hold for all $\eps$
such that $0 < \eps < \eps'$.
Then, for every such $\eps$, the following facts are true.
Let $X'_\eps := X \ssminus \varphi\Inv(B_\eps)$.
\begin{itemize}
\item[(a)]
For every two points $y_0$ and $y_1$ in $X'_\eps \times \Rplus$,
if the segment $[\Psi(y_0),\Psi(y_1)]$
does not contain the origin,
then there exists a path in $X'_\eps \times \Rplus$
connecting $y_0$ to $y_1$ whose composition with $\Psi$
is weakly monotone straight.
\item[(b)]
The image $\Psi(X'_\eps \times \Rplus)$ is equal
to $C \ssminus \{ 0 \}$.
\item[(c)]
The map $\Psi|_{X'_\eps \times \Rplus}$ is open as a map to $C$.
\end{itemize}

Because
the space $ X \ssminus \varphi\Inv(0) $ is the union of the open subsets
$X'_\eps$, for $0 < \eps < \eps'$,
and by the fact (c), 
the map $\Psi|_{ ( X \ssminus \varphi\Inv(0) ) \times \Rplus}$
is open as a map to~$C$.

Choose any $x \in \varphi\Inv(0)$.
Consider the set $\Psi( (U_x \ssminus \varphi\Inv(0)) \times \Rplus)$.
This set is equal to $C_x \ssminus \{ 0 \}$,
because $C_x$ is a cone with vertex at the origin
and $\varphi(U_x)$ is open in $C_x$ and contains the origin.
But this set is also open in $C$,
because $\Psi|_{ (X \ssminus \varphi\Inv(0)) \times \Rplus}$
is open as a map to $C$.   So $C_x \ssminus \{ 0 \}$ is open in $C$.
Because $C_x$ is a closed convex cone and $C$ is connected,
it follows that $C_x$ is equal to $C$.
This proves Claim (2).

In particular, since $C_x = C$, 
we now know that the restriction of $\varphi$ to $U_x$
is open as a map to $C$.
This implies that the restriction of $\Psi$ to $U_x \times \Rplus$
is also open as a map to $C$.
This is true for every $x \in \varphi\Inv(0)$.
But we also showed that the restriction of $\Psi$
to $(X \ssminus \varphi\Inv(0)) \times \Rplus$ is open as a map to $C$.
Because the space $X$ is the union of the open subsets $U_x$
for $x \in \varphi\Inv(0)$ and the open subset $X \ssminus \varphi\Inv(0)$,
we obtain Claim (3).

Claim (1), in the case that $0 \not\in [\Psi(y_0),\Psi(y_1)]$,
follows from the above fact (a)
when $\eps$ is chosen small enough
so that both $y_0$ and $y_1$ lie in $X'_\eps \times \Rplus$.

Suppose that $\Psi(y_0)$ and $\Psi(y_1)$ are both nonzero
but $0 \in [\Psi(y_0),\Psi(y_1)]$.
Choose any $y' \in \Psi\Inv(0)$.
If we can find paths from $y_0$ to $y'$ and from $y'$ to $y_1$
whose compositions with $\Psi$ are weakly monotone straight paths,
their concatenation will give a path from $y_0$ to $y_1$
whose composition with $\Psi$ is a weakly monotone straight path.

It remains to prove Claim (1) in the case that, say,
$\Psi(y_0) = 0$ and $\Psi(y_1) \neq 0$.
Fix such $y_0$ and $y_1$.
Write $y_0 = (x,\lambda)$; then $x$ is in $\varphi\Inv(0)$.
By the definition of $C$ and because $C=C_x$,
the value $\Psi(y_1)$ belongs to $C_x$.
Because $\varphi(U_x)$ is an open subset of $C_x$
that contains the origin,
if $\eps$ is a sufficiently small positive number,
then $\eps \Psi(y_1)$ belongs to $\varphi(U_x)$.
Fix such an $\eps$, and let $x'$ be a point in $U_x$
such that $\varphi(x') = \eps \Psi(y_1)$.
Because $\varphi|_{U_x}$ is convex,
there exists a path $x(\cdot)$ in $U_x$ from $x$ to $x'$
whose composition with $\varphi$ is a weakly monotone straight path
from the origin to $\eps \Psi(y_1)$.

Then $(x(\cdot),1)$ is a path from $(x,1)$ to $(x',1)$
whose composition with $\Psi$ is a weakly monotone straight path
from the origin to $\eps \Psi(y_1)$.
Concatenating with a path from $y_0 = (x,\lambda)$ to $(x,1)$
which is entirely contained in $\{ x \} \times \Rplus$,
we obtain a path $\gamma'$ from $y_0$ to $(x',1)$
whose composition with $\Psi$ is a weakly monotone straight path
from the origin to $\eps \Psi(y_1)$.

By the case of Claim (1) that we already proved,
there exists a path $\gamma''$ from $(x',1)$ to $y_1$
whose composition with $\Psi$ is a weakly monotone straight path
from $\eps \Psi(y_1)$ to $\Psi(y_1)$.
The concatenation of $\gamma'$ with $\gamma''$
is a path from $y_0$ to $y_1$
whose composition with $\Psi$ is weakly monotone straight.

This completes the proof of Claim (1)
and of Proposition~\ref{with zero'}.
\end{proof}

\section{Linear maps on the simplex}
\labell{sec:linear}

We begin by setting some notation:  
\begin{gather*}
\R_+^n = \{ (s_1,\ldots,s_n) \in \R^n \ | \ s_j \geq 0
\text{ for } j=1,\ldots, n \}; \\
\Lambda = \{ (s_1,\ldots,s_n) \in \R_+^n \ | \ s_1 + \ldots + s_n < 1 \}.
\end{gather*}
The goal of this section is to prove the following proposition.

\begin{Proposition} \labell{orthant projection}
Let $L \co \R^n \to \R^k$ be a linear map, 
and let $A = L(\Lambda)$.  Then the map
$$ L|_{\Lambda} \co \Lambda \to A $$
is open, has the weak path lifting property
(cf.\ Definition~\ref{path lifting}),
and its level sets are path connected.
\end{Proposition}

\begin{Remark} \labell{projection not open}
The restriction of a linear projection
to a convex set is \emph{not} necessarily open
as a map to its image.
For example, let $D_r \subset \R^2$ denote
the closed disc of radius $r$ and center $(r,0)$,
and let $X = \{ (x,y,r) \in \R^3 \ | \ (x,y) \in D_r \}$.
Then $X$ is convex, and the restriction to $X$ 
of the projection $(x,y,z) \mapsto (x,y)$ 
is not open as a map to its image.
\end{Remark}

We will need the following variant of Carath\'{e}odory's theorem
from convex geometry:

\begin{Lemma}\labell{subconvex}
Let $v_1, \dots, v_n$ be vectors in $\R^k$. 
Let $w=\sum s_jv_j$ with $(s_1,\ldots,s_n) \in \Lambda$.
Then there exists $(s'_1,\ldots,s'_n) \in \Lambda$
such that $w=\sum s'_jv_j$ and such that the vectors
$v_j$ for which $s'_j \neq 0$ are linearly independent.
\end{Lemma}

\begin{proof}
We prove the lemma by induction on $n$.
If $n=1$, the lemma is obvious.
Suppose that the lemma is true for $n-1$ vectors,
and we will prove it for $n$ vectors.

Suppose that $w = \sum_{j=1}^n s_j v_j$
with $(s_1,\ldots,s_n) \in \Lambda$
and that $v_1,\ldots,v_n$ are not linearly independent.
If one or more of $s_1,\ldots,s_n$ is zero,
the required conclusion follows from the induction hypothesis.
Suppose that $s_1,\ldots,s_n$ are all positive.

Then there exist $\lambda_1, \dots, \lambda_n$ not all zero such that 
$\sum \lambda_j \geq 0$ and $\sum \lambda_j v_j =0$. 
At least one of the $\lambda_j$s is positive. 
Choose $c=\min\{\frac{s_j}{\lambda_j}\mid \lambda_j
>0\}=\frac{s_i}{\lambda_i}$. Then
$$ w = \sum s_jv_j -c\sum \lambda_jv_j=\sum (s_j-c\lambda_j) v_j.$$
Note that $s_j-c\lambda_j \geq 0$ for all $j$, and $s_i-c\lambda_i = 0$. 
Furthermore, since $c >0$ and $\sum \lambda_j \geq 0$,
\[
\sum (s_j - c\lambda_j) =\sum s_j - c \sum \lambda_j \leq \sum s_j < 1.
\] 
The required conclusion now follows from the induction hypothesis
applied to the vectors $\{ v_1, \ldots, v_n \} \ssminus \{ v_i \}$.
\end{proof}

\begin{Lemma} \labell{orthant loc open}
Let $L \co \R^n \to \R^k$ be a linear map,
and let $C = L(\R_+^n)$.
Then $L(\Lambda)$ is a relative neighborhood of $0$ in $C$.
\end{Lemma}

\begin{proof}
Let $e_1,\ldots,e_n$ denote the standard basis of $\R^n$,
and let $v_i = L(e_i)$ for $i=1,\ldots,n$. 
Let $\calJ$ denote the set of subsets $J \subset \{ 1,\ldots,n \}$
for which the vectors $v_j$, for $j \in J$, are linearly independent.

For each $J \in \calJ$, let
$ \R_+^J = \{ s \in \R_+^n \ | \ s_j = 0
   \text{ for all } j \not\in J \}$,
and let $C_J = L(\R_+^J)$.
Also, for $\eps > 0$, let $B_\eps$ denote the ball in $\R^k$
of radius $\eps$ centered at the origin.

For every $J \in \calJ$, the map $L$ restricts to a homeomorphism
from $\R_+^J$ to $C_J$. So $L(\Lambda\cap \R_+^J)$ is open in $C_J$
and contains the origin.
Let $\eps_J$ be a positive number such that $L(\Lambda \cap \R_+^J)$
contains $C_J \cap B_{\eps_J}$.
By Lemma~\ref{subconvex}, the union of the sets $C_J$,
for $J \in \calJ$, is all of $C$.
Let $\eps = \min_{J \in \calJ} \eps_J$.
Then $L(\Lambda)$ contains $C \cap B_\eps$.
\end{proof}

\begin{Lemma} \labell{sector loc open}
Let $L \co \R^n \to \R^k$ be a linear map,
let $0 \leq r \leq n$, let $S^n_r = \R_+^r \times \R^{n-r}$
be a sector, and let $C = L(S^n_r)$ be its image.
Let $\calO$ be a neighborhood of the origin in $S^n_r$.
Then $L(\calO)$ is a relative neighborhood of the origin in $C$.
\end{Lemma}

\begin{proof}
For every 
$o = (o_1,\ldots,o_n) \in \{ 1 \}^r \times \{ -1, 1 \}^{n-r}$, 
let
$F_o \co \R^n \to \R^n$ be the map 
$$ F_o(s_1,\ldots,s_n) = (o_1 s_1 , \ldots , o_n s_n) .$$
The cone $C$ is the union, 
over all $o \in \{ 1 \}^r \times \{ -1, 1\}^{n-r} $, 
of the sets 
$$ C_o := L (F_o (\R_+^n) ).$$ 

Let $\rho_o$ be a positive number
such that $F_o(\rho_o\Lambda)$ is contained in $\calO$.
It exists because $F_o|_{\R_+^n}$ is continuous
and carries $0$ into the open set $\calO$,
and because every neighborhood of $0$ in $\R_+^n$
contains a set of the form $\rho\Lambda$ for some $\rho > 0$.

By Lemma~\ref{orthant loc open}, applied to the linear map
$x \mapsto  L (F_o(\rho_0 x))$,
there exists $\eps_o > 0$
such that $L(F_o(\rho_o \Lambda))$
contains $B_{\eps_o} \cap C_o$.
Let $\eps = \min_{o \in O} \eps_o$.
Then $L(\calO)$ contains $B_\eps \cap C$.
\end{proof}

\begin{Lemma} \labell{orthant open}
Let $L \co \R^n \to \R^k$ be a linear map,
and let $C = L(\R_+^n)$. Then the map
$$ L|_{\R_+^n} \co \R_+^n \to C $$
is open.
\end{Lemma}

\begin{proof}
Let $x \in \R_+^n$.  Without loss of generality
we may assume that $x_j = 0$ for all $1 \leq j \leq r$
and $x_j > 0$ for all $r+1 \leq j \leq n$.

Let $\calO$ be a sufficiently small neighborhood of the origin 
in the sector $S^n_r$
so that the translation $x + \calO$ is
contained in $\R_+^n$.
Then $x + \calO$ is a neighborhood of $x$ in $\R_+^n$.

By Lemma~\ref{sector loc open}, $L(\calO)$ is a neighborhood
of the origin in $L(S^n_r)$. 
This implies that $L(x+\calO)$ is a neighborhood 
of $L(x)$ in $C$.  Indeed, let $\eps > 0$ be such that 
$L(\calO)$ contains $B_\eps \cap L(S^n_r)$.
Then $L(x+\calO)$ contains the $\eps$--neighborhood
of $L(x)$ in $C$.  This follows from the fact that,
for every $y \in C$, the difference $y-L(x)$ is in $L(S^n_r)$. 

This shows that the map $L|_{\R_+^n} \co \R_+^n \to C$
is open.  
\end{proof}

The closure of $\Lambda$ in $\R^n$ is the simplex
$$ \Lambdabar = \{ s \in \R_+^n \ | \ s_1 + \ldots + s_n \leq 1 \}.$$

\begin{Lemma} \labell{projection to Aext}
Let $L \co \R^n \to \R^k$ be a linear map,
$A = L(\Lambda)$, $\Abar = L(\Lambdabar)$, and $A^\ext = \Abar \ssminus A$.
Then
\begin{enumerate}
\item
For every $\beta \in A \ssminus \{ 0 \}$, 
the intersection $\R_+ \beta \cap A^\ext$
contains exactly one point; call it $\beta^\ext$.
\item
$\beta \mapsto \beta^\ext$ defines a continuous map 
from $A \ssminus \{ 0 \}$ to $A^\ext$.
\end{enumerate}
\end{Lemma}

\begin{Remark} The notation ``ext" stands for ``extremal". \end{Remark}

\begin{proof}[Proof of Lemma~\ref{projection to Aext}]
Fix $\beta \in A \ssminus \{ 0 \}$.
Because the subset $\Abar$ of $\R^k$ is closed, bounded, and contains $\beta$,
the set
\begin{equation} \labell{set of t}
 \{ t \in (0,1] \ | \ \frac{1}{t} \beta \in \Abar \} 
\end{equation}
has a positive minimum; call it $t_\beta$.
Note that
$\frac{1}{t_\beta} \beta$ is in $\R_+ \cdot \beta \cap \Abar$.
So, for (1), it is enough to show that $\frac{1}{t_\beta} \beta$ is not in $A$.
Equivalently, it is enough to show that if $\frac{1}{t} \beta$ is in $A$,
then $t$ is not minimal in~\eqref{set of t}.  Suppose now that 
$\frac{1}{t} \beta$ is in $A$.  Write it as $\sum s_j v_j$ where 
the coefficients $s_j$ are nonnegative and where $\sum s_j < 1$.
Then $(\sum s_j) t$ is strictly smaller than $t$
and $\frac{1}{(\sum s_j) t} \beta$ is in $\Abar$,
so $t$ is not minimal in~\eqref{set of t}.  This proves (1).

To prove (2),
suppose that $\beta_n$ is a sequence of elements
of $A \ssminus \{ 0 \}$, let $\beta_n^\ext$ be their images
in $A^\ext$, suppose that the sequence $\beta_n$
converges to an element $\beta_\infty$ of $A \ssminus \{ 0 \}$,
and 
suppose that the sequence $\beta_n^\ext$ converges to 
an element $\beta'$ of $\R^k$.

Lemma~\ref{orthant open} implies that $A$ is open in $L(\R_+^n)$,
and hence in $\Abar$.  This, in turn, implies that $A^\ext$ is closed 
in $\Abar$, and hence in $\R^k$.   Thus, $\beta'$ must be in $A^\ext$.
In particular, $\beta'$ is nonzero.

Because $\beta_n^\ext \in \R_+ \beta_n$
and $\begin{CD} \beta_n @> \phantom{n} > n \to \infty > \beta_\infty \end{CD}$,
we have
$\begin{CD} 
\frac{\beta_n^\ext}{\| \beta_n^\ext \|}
@> \phantom{n} > n \to \infty > \frac{\beta_\infty}{\| \beta_\infty \|}.
\end{CD} $
Because 
$\begin{CD} \beta_n^\ext @> \phantom{n} > n \to \infty > \beta' \end{CD}$,
we have
$\begin{CD} 
\frac{\beta_n^\ext}{\| \beta_n^\ext \|}
@> \phantom{n} > n \to \infty > \frac{\beta'}{\| \beta' \|}. \end{CD}$
By uniqueness of the limit, we deduce that $\beta' \in \R_+ \beta_\infty$.
Because $A^\ext$ intersects every ray in at most one point,
we must have $\beta' = \beta^\ext_\infty$.

Now suppose that $\beta_n$ is any sequence of elements
of $A \ssminus \{ 0 \}$ that converges to an element $\beta_\infty$
of $A \ssminus \{ 0 \}$.  The above argument implies that
$\beta_\infty^\ext$ is the limit of every converging subsequence
of $\beta_n^\ext$.  Because $\beta_n^\ext$ are in $A^\ext$
and the set $A^\ext$ is compact, this implies that
the sequence $\beta_n^\ext$ converges to $\beta_\infty^\ext$.
\end{proof}

\begin{Lemma} \labell{exists section}
Let $L \co \R^n \to \R^k$ be a linear map.
Then $L|_\Lambda$ has a continuous section.
That is, there exists a continuous map $\sigma \co A \to \Lambda$,
where $A = L(\Lambda)$, such that $L \circ \sigma = \text{id}_A$.
\end{Lemma}

\begin{proof}
First, we show that the map
$$ L|_{\Lambdabar} \co \Lambdabar \to \Abar $$
has a continuous section:
$$ \sigmabar \co \Abar \to \Lambdabar \quad , \quad 
   L \circ \sigmabar = \text{id}_{\Abar}.$$
We define $\sigmabar$ recursively on the faces of $\Abar$.
First, we define it, arbitrarily, on vertices. 
Now, suppose that $Q$ is a face of $\Abar$
and that we already defined $\sigmabar$ on the relative boundary
$\del Q$ of $Q$.  We define $\sigmabar$ arbitrarily at a point $q$
in the relative interior of $Q$, and we extend it 
in an affine manner on segments connecting $q$ to $\del Q$.

Next, we restrict this section to the closed subset $A^\ext$
to obtain a continuous map
$$ \sigma^\ext \co A^\ext \to \Lambdabar $$
such that $L \circ \sigma^\ext = \text{id}_{A^\ext}$.

Finally, suppose that $\beta \in A \ssminus \{ 0 \}$
and let $\beta^\ext$ be its image in $A^\ext$.
Then $\beta = t_\beta \beta^\ext$.
A priori $t_\beta \in (0,1]$, but, 
because $\beta$ itself is not in $A^\ext$,
the number $t_\beta$ is strictly less than one.
We then define
$\sigma(\beta) = t_\beta \sigma^\ext(\beta^\ext)$.

Because $\sigma(\beta)$ is the product of an element
of $\Lambdabar$ with a positive number that is strictly 
less than one, it is in $\Lambda$. 
So $\sigma$ is a map from $A$ to $\Lambda$.

By Lemma~\ref{projection to Aext}, the map $\beta \mapsto \beta^\ext$
is continuous; it follows that $\beta \mapsto t_\beta$ is also continuous.  
Thus, $\sigma \co A \to \Lambda$ is continuous.

Finally, 
$L(\sigma(\beta)) = L( t_\beta \sigma^\ext(\beta^\ext) )
 = t_\beta \beta^\ext = \beta$.
So $L \circ \sigma = \text{id}_A$,
and $\sigma$ is a section of $L|_\Lambda \co \Lambda \to A$, as required.
\end{proof}

\begin{proof}[Proof of Proposition~\ref{orthant projection}]
By Lemma~\ref{orthant open}, 
the map $L|_{\R_+^n}$ is open as a map to its image.
Because $\Lambda$ is open in $\R_+^n$, this implies that
the map $L|_\Lambda$ is also open as a map to its image.

Because $\Lambda$ is convex and $L$ is linear, the level sets
of $L|_\Lambda$ are path connected.

By Lemma~\ref{exists section}, the map $L|_\Lambda \co \Lambda \to A$
has a continuous section.  This together with the connectedness of the 
level sets implies the weak path lifting property.
\end{proof}

\begin{Corollary}\labell{model local}
Let $v_1, \dots, v_n$ 
be vectors in $\R^k$. Let $B_\eps$ be an open $\eps$--ball
about $0$ in $\R^\ell$. Let
\[
\Lambda = \{s\in \R^n \mid s_j \geq 0 \text{ for } j=1, \dots, n, 
   \text{ and }\sum_{j=1}^n s_j < 1\}.
\]
Let
\[
A = \{\sum_{j=1}^n s_jv_j \mid s_j \geq 0 \text{ for } j=1, \dots, n, 
   \text{ and }\sum_{j=1}^n s_j < 1\},
\]
Then the map
\[
\varphi\co \Lambda \times B_\eps \to A \times B_\eps
\quad \text{given by} \quad
\left((s_1, \dots, s_n), \eta\right) \mapsto 
 \left(\sum_{j=1}^n s_jv_j, \eta\right) \]
is open, has the weak lifting property,
and its level sets are path connected.
\end{Corollary}

\begin{proof}
These properties follow from the analogous
properties of the map $s \mapsto \sum s_j v_j$
from $\Lambda$ to $A$, which, 
in turn, were established in Proposition~\ref{orthant projection}.
\end{proof}

\section{Contact momentum maps}
\labell{sec:contact}

An \emph{exact symplectic manifold} is a symplectic manifold $(Q,\omega)$
such that the symplectic form $\omega$ is exact:  there exists
a one--form $\alpha$ such that $\omega = d\alpha$.  Let a torus $T$
act on an exact symplectic manifold $(Q,\omega)$ and, for every
Lie algebra element $X \in \ft$, let $X_Q$ be the corresponding
vector field on $Q$.  Suppose that $\omega = d\alpha$
and that $\alpha$ is $T$--invariant.
Then the map $\Phi \co Q \to \ft^*$ given by
\begin{equation}\labell{E:exact mm}
	\Phi^X(q) = \alpha(X_Q(q))
\end{equation} for all $X \in \ft$ and $q \in Q$
is a momentum map:  $d\Phi^X = -\iota(X_Q)\omega$.
An \emph{exact momentum map} is a momentum map
that has this form.

We recall from Section~\ref{sec:intro} that, if a torus $T$ 
acts on a manifold $M$ and preserves a contact one--form $\alpha$,
the \emph{$\alpha$--momentum map} is the map $\Psi_\alpha \co M \to \ft^*$
defined by $\Psi_\alpha^X = \iota(X_M) \alpha$,
and the \emph{contact momentum map} is the map
$\Psi \co M \times \Rplus \to \ft^*$ defined by
$\Psi(x,t) = t \Psi_\alpha(x)$.
The map $\Psi$ is an exact momentum map on the symplectization
$(M \times \Rplus, d(t\alpha) )$,
with the trivial extension of the $T$ action to $M \times \Rplus$,
and the map $\Psi_\alpha$ is a momentum map for the closed
(degenerate) two--form $d\alpha$ on $M$.

\begin{Lemma} \labell{exact mm}
Let a torus $T$ act on an exact symplectic manifold $(Q,\omega)$
with exact momentum map $\Phi \co Q \to \ft^*$.
Let $q$ be a point of $Q$.
Then every neighborhood of $q$ in $Q$ contains a smaller neighborhood $U$
such that
\begin{enumerate}
\item
The map $\Phi|_U \co U \to \Phi(U)$ is open, 
has the weak path lifting property (cf.~Definition~\ref{path lifting}),
and its level sets are path connected.
\item
There exists a convex polyhedral cone $C_q$ with vertex at the origin,
such that the set $\Phi(U)$ 
is a relatively open subset of $C_q$.
\end{enumerate}

\end{Lemma}

\begin{proof}
Let $H$ denote the stabilizer of $q$.
Let $\fh$ denote its Lie algebra, $\fh^*$ the dual space,
and $\fh^0$ the annihilator of $\fh$ in $\ft^*$.
Fix an inner product on $\ft$ and use it to 
identify $\fh^*$ with a subspace of $\ft^*$.
By the local normal form theorem for Hamiltonian torus actions,
there exists an action of $H$ on $\C^n$,
with weights $\eta_1,\ldots,\eta_n \in \fh^*$,
there exists a $T$--invariant symplectic form  
on the model $Y := T \times_H \C^n \times \fh^0$
(with the left $T$ action),
and there exists an equivariant symplectomorphism
from an invariant neighborhood of $q$ in $Q$ 
to an open subset of $Y$
that carries the point $q$ to $[1,0,0]$
and that carries the momentum map $\Phi$ to the map 
\begin{equation} \labell{PhiY}
\Phi_Y([t,z,\nu]) \, = \, \Phi(q) \, + \,
\frac{|z_1|^2}{2} \eta_1 + \ldots + \frac{|z_n|^2}{2} \eta_n
 \, + \, \nu.
\end{equation}

The image of $\Phi_Y$ is the translation of the cone
$$ C_q = \{ s_1 \eta_1 + \ldots + s_n \eta_n + \nu \ | \ 
s_j \geq 0 \text{ for all } j,
\text{ and } \nu \in \fh^0 \} $$
by the element $\Phi(q)$ of $\ft^*$.
The cone $C_q$ is a convex polyhedral cone in $\ft^*$,
with vertex at the origin,
invariant under translations by elements of $\fh^0$.
By the formula~\eqref{E:exact mm} for the exact momentum map,
the element $\Phi(q)$ of $\ft^*$ 
belongs to the annihilator $\fh^0$ of $\fh$.
It follows that the image of $\Phi_Y$ is equal to $C_q$.

We need to show that
the restriction of $\Phi$ to arbitrarily small neighborhoods $U$
of $q$ in $Q$ satisfies (1) and (2).
By the local normal form theorem, it is enough 
to show these properties
for the restriction of $\Phi_Y$ 
to neighborhoods of $[1,0,0]$ in $Y$.

For $\eps>0$, let $B_\eps^T$ be the $\eps$--neighborhood
of the unit element in $T$ (with respect to some invariant metric);
so $B_\eps^T \cdot H$ is the $\eps$--neighborhood of $H$ in $T$;
let $B_\eps^{\C^n}$ be the $\eps$--ball about the origin in $\C^n$;
and let $B_\eps^{\fh^0}$ be the $\eps$--ball about the origin in $\fh^0$.
Let
$$U_\eps :=
  (B_\eps^T \cdot H) \times_H B_\eps^{\C^n} \times B_\eps^{\fh^0}.$$
Because every neighborhood of $[1,0,0]$ in $Y$ contains
a set of the form $U_\eps$
for some $\eps > 0$, it is enough to show that
\begin{enumerate}
\item[(1')]
The map $\Phi_Y|_{U_\eps} \co U_\eps \to \Phi_Y(U_\eps)$
is open, has the weak path lifting property,
and its level sets are path connected.
\item[(2')]
The set $\Phi_Y(U_\eps)$ is a neighborhood of $\Phi(q)$
in the cone $C_q$.
\end{enumerate}

The map $(z_1,\ldots,z_n) \mapsto (s_1,\ldots,s_n)$,
where $|z_j|^2 = \eps^2 s_j$,
takes $B_\eps^{\C^n}$ onto the set
$$ \Lambda := \{ s \in \R_+^n \ | \ s_1 + \ldots + s_n < 1 \} .$$
By~\eqref{PhiY}, 
$$ \Phi_Y(U_\eps) = \{ \Phi(q) + \frac{\eps^2}{2} \sum s_j \eta_j + \nu
 \ | \ (s_1,\ldots,s_n) \in \Lambda \text{ and } \nu \in B_\eps^{\fh^0} \}.$$
The affine isomorphism
$$ (\beta,\nu) \mapsto \Phi(q) + \frac{\eps^2}{2}\beta + \nu $$
of $\fh^* \times \fh^0$ with $\ft^*$
carries the cone $\{ \sum s_j \eta_j \ | \ s_j \geq 0
 \text{ for all } j \} \times \fh^0$
to the cone $C_q$ and the subset 
$A \times B_\eps^{\fh^0} $,
where
$$ A := \{ \sum s_j \eta_j \ | \ (s_1,\ldots,s_n) \in \Lambda \}, $$
to $\Phi_Y(U_\eps)$.
Because, by Lemma~\ref{orthant loc open}, the set $A$
is open in the cone
$\{ \sum s_j \eta_j \ | \ s_j \geq 0 \text{ for all } j  \}$,
this gives (2').

The map
$$ \begin{CD}
B_\eps^T \times \Lambda \times (S^1)^n \times B_\eps^{\fh^0}
 @>>>  U_\eps
\end{CD} $$
given by
$$ \begin{CD}
 \left( \lambda,(s_1,\ldots,s_n),(e^{i\theta_1},\ldots,e^{i\theta_n}),\nu 
 \right)
     \ \mapsto \ [\lambda,z,\nu] 
\qquad \text{ where } z_j = \eps \sqrt{s_j} e^{i\theta_j} 
\end{CD}$$
is continuous and onto.  
So, for (1'), it is enough to show that the composition
of this map with $\Phi_Y$, as a map to $\Phi_Y(U_\eps)$,
is open, has the weak path lifting property,
and its level sets are path connected.

This composition can be expressed as the composition of the map
\begin{equation} \labell{map1}
\begin{CD}
 B_\eps^T \times \Lambda \times (S^1)^n \times B_\eps^{\fh^0}
 @> \text{projection} >> \Lambda \times B_\eps^{\fh^0}
\end{CD}
\end{equation}
with the map
\begin{equation} \labell{map2}
\begin{CD}
\Lambda \times B_\eps^{\fh^0}
 @> (s,\nu) \mapsto (\sum s_j \eta_j,\nu) >>
A \times B_\eps^{\fh^0}
\end{CD}
\end{equation}
and the map
\begin{equation} \labell{map3}
\begin{CD}
A \times B_\eps^{\fh^0}
 @> (\beta,\nu) \mapsto \Phi(q) + \frac{\eps^2}{2}\beta+\nu >>
\Phi_Y(U_\eps).
\end{CD}
\end{equation}
So, for (1'), it is enough to check that each of the maps~\eqref{map1},
\eqref{map2}, and~\eqref{map3} is continuous, open, onto,
has the weak path lifting property, and its level sets are path connected.  
The map~\eqref{map1} is a fibration with path connected fibers,
and the map~\eqref{map3} is a homeomorphism,
so they both have the required properties.
The required properties of the map~\eqref{map2} 
follow from Corollary~\ref{model local}.

\end{proof}

\begin{Lemma}\labell{Psialpha local}
Let a torus $T$ act on a contact manifold $(M, \xi=\ker\alpha)$ 
with $\alpha$--momentum map $\Psi_\alpha\co M \to \ft^*$.  
Then, for every point $x$ of $M$ with $\Psi_\alpha(x)=0$, there exists 
an open neighborhood $U_x$ of $x$ in $M$ and a convex polyhedral
 cone $C_x$ 
in $\ft^*$ with vertex at the origin such that
\begin{itemize}
	\item The image $\Psi_\alpha(U_x)$ is an open subset of $C_x$.
	\item The map $\Psi_\alpha|_{U_x}\co U_x \to \Psi_\alpha(U_x)$ 
is open, has the weak path lifting property,
and its level sets are path connected.
\end{itemize}

\end{Lemma}

\begin{proof}
Let $x$ be a point in $M$ with $\Psi_\alpha(x)=0$.
Let $R_\alpha$ be the Reeb vector field of the contact form $\alpha$.
Recall that it is defined by the conditions
$\iota(R_\alpha) d\alpha = 0$ and $\alpha(R_\alpha) = 1$.
The null space of $d\alpha|_{T_xM}$ is $\R R_\alpha(x)$.
Because $x$ is in the zero level set of the $\alpha$--momentum map,
the tangent space to its orbit, $T_x(T \cdot x)$, 
is contained in the contact distribution,
so $\R R_\alpha(x) \cap T_x(T \cdot x) = \{ 0 \}$.

Let $H$ be the stabilizer of $x$; it acts linearly on $T_xM$.
Let $W$ be an $H$--invariant subspace of $T_xM$ that is complementary 
to $\R R_\alpha(x) \oplus T_x (T \cdot x)$. 
Then we have an $H$--invariant decomposition
$$ T_xM = \R R_\alpha(x) \oplus T_x (T \cdot x) \oplus W, $$
and $d\alpha$ is nondegenerate on $T_x(T \cdot x) \oplus W$.

Let $\psi$ be an $H$--equivariant diffeomorphism
from a neighborhood of the origin in $T_xM$
to a neighborhood of $x$ in $M$
whose differential at $x$ is the identity map on $T_xM$.

Denote the Reeb trajectory of a point $q$ by $q^{(t)}$.
Thus, $q^{(0)} = q$ and $\frac{d}{dt} q^{(t)} = R_\alpha(q^{(t)})$.
Then, for the interval $I=(-\eps,\eps)$ with sufficiently small $\eps>0$
and for a sufficiently small neighborhood $D$ of the origin in $W$,
the formula $(t , [a,u]) \mapsto ( a \cdot \psi(u) )^{(t)} $
defines a diffeomorphism from $I \times (T \times_H D)$
to an open subset of $M$,
and the image of $\{ 0 \} \times (T \times_H D)$ under this diffeomorphism
is a submanifold of $M$ on which $d\alpha$ is nondegenerate.

We denote this submanifold by $Q$, the inclusion map by $i \co Q \to M$,
the symplectic form by $\omega_Q := i^*d\alpha$,
and the momentum map by $\Phi_Q := i^* \Psi_\alpha$.

The map $f \co I \times Q \to M$, given by $(t,q) \mapsto q^{(t)}$,
is a $T$--equivariant diffeomorphism to an invariant open subset of $M$,
and the pullback of $d\alpha$ through this diffeomorphism
is equal to the pullback of $\omega_Q$ with respect to the projection map
$I \times Q \to Q$.
It follows that the pullback $f^*\Psi_\alpha$ must have the form
$(t,q) \mapsto \Phi_Q(q)$.
The properties of $\Psi_\alpha$ then follow
from the corresponding properties of $\Phi_Q$, 
which are guaranteed by Lemma~\ref{exact mm}.
\end{proof}

\begin{Remark}
For a contact manifold with compact group action,
Frank Loose~\cite[Theorem~3]{loose} gives a local normal form
that describes the neighborhood of an orbit in the zero level set,
up to equivariant contactomorphism.
(Without a contact one--form, he defines a momentum map 
with values in $\fg^* \otimes L$, where $L$ is the line
bundle $TM/\xi$ over $M$.)
\end{Remark}

\begin{Lemma}\labell{Psi local}
	Let a torus $T$ act on a contact manifold $(M, \xi=\ker\alpha)$ 
	with $\alpha$--momentum map $\Psi_\alpha\co M \to \ft^*$
	and contact momentum map $\Psi\co M\times \Rplus \to \ft^*$.
	
	Choose a metric on $\ft^*$ and let $S(\ft^*)$ denote the unit sphere 
	in $\ft^*$. Define the map 
	$\Psibar\co (M\ssminus \Psi_\alpha\Inv(0))\times \Rplus \to S(\ft^*)$ 
	by $\Psibar=\Psi/\|\Psi\|$.
	
	Then, for any $(x, \lambda)\in M \times \Rplus$ with $\Psi_\alpha(x)\neq 0$, 
	every neighborhood of $(x, \lambda)$ in $M \times \Rplus$ 
	contains a smaller neighborhood $U$ such that 
	\begin{itemize}
		\item The map $\Psi|_U$ is convex.  
		\item The image $\Psi(U)$ is a relatively open subset in a convex polyhedral 
		cone with vertex at the origin.
		\item The map $\Psibar|_U$ 
		is (defined and) open as a map to its image.
	\end{itemize}
\end{Lemma}	

\begin{proof}
Let $x\in M\ssminus \Psi_\alpha\Inv(0)$. 
Because $\Psi$ is an exact momentum map, 
by Lemma~\ref{exact mm}, every neighborhood of $(x, \lambda)$ 
in $M\times \Rplus$ contains a smaller neighborhood $U'$ 
such that the map $\Psi|_{U'} \co U' \to \Psi(U')$ is open, 
has the weak path lifting property, its level sets are path connected,
and its image has the form $\calO \cap C_{(x,\lambda)}$ 
where $\calO$ is an open neighborhood of $\Psi(x,\lambda)$
and $C_{(x,\lambda)}$ is a convex polyhedral cone 
with vertex at the origin. We may assume $\calO$ does not contain the origin.
Let $B$ be a convex open neighborhood of $\Psi(x,\lambda)$
that is contained in $\calO$,
and let $U$ be the intersection of $U'$ with the preimage of $B$.
Then $\Psi|_U \co U \to \Psi(U)$ still has the weak path lifting
property and its level sets are path connected, but, additionally,
its image, being equal to $B \cap C_{(x,\lambda)}$, is convex.
These properties imply that $\Psi|_U$ is convex. 
Since $\Psi|_U\co U \to \Psi(U)$ is still open 
and the image $\Psi(U) = B \cap C_{(x, \lambda)}$ is open 
in $C_{(x, \lambda)}$, 
the map 
$\Psi|_U\co U \to C_{(x, \lambda)}$ is a continuous nonvanishing open map. 
By Lemma \ref{to cone}, the map 
$\Psibar|_U\co U \to C_{(x, \lambda)} \cap S(\ft^*)$ is open. 
This implies that $\Psibar|_U$ is open as a map to its image.        	
\end{proof}

\begin{Lemma} \labell{convex polyhedral}
Let $C$ be a closed convex cone in $\R^n$ with vertex at the origin.
Suppose that for every $w$ in $C \ssminus \{ 0 \}$ 
there exists a neighborhood $U_w$ in $\R^n$ 
and a convex polyhedral cone $C_w$ with vertex at the origin
such that $U_w \cap C = U_w \cap C_w$.
Then $C$ is a convex polyhedral cone with vertex at the origin.
\end{Lemma}

\begin{proof}
After possibly shrinking the neighborhoods $U_w$, we may assume
that these neighborhoods are convex.  
We may also assume that $\Rplus \cdot U_w = U_w$.
Otherwise, we replace $U_w$ by $\Rplus \cdot U_w$;
it remains an open set that satisfies $U_w \cap C = U_w \cap C_w$.

Recall that a convex polyhedral cone with vertex at the origin
is a finite intersection of closed half--spaces whose boundaries
contain the origin.

For $w \in C$,
let $H_w^j$, for $1 \leq j \leq N_w$, be closed half--spaces
whose boundaries contain the origin
and such that $C_w = H_w^1 \cap \ldots \cap H_w^{N_w}$.
(It is possible that $N_w=0$ and $C_w = \R^n$.)
Let $\del H_w^j$ denote the boundary of $H_w^j$.
We may assume that $U_w \cap C \cap \del H_w^j \neq \varnothing$
for every $1 \leq j \leq N_w$.  Otherwise, we replace
the polyhedral cone $C_w$ by the intersection of those $H_w^j$ 
that do satisfy $U_w \cap C \cap \del H_w^j \neq \varnothing$;
this intersection is a (possibly larger) cone 
that still satisfies $U_w \cap C = U_w \cap C_w$.

Because $C \cap S^{n-1}$ is compact,
we may choose a finite set of points 
$W = \{ w_1,\ldots,w_N \} \subset C \cap S^{n-1}$
such that $U_{w_1} \cup \ldots \cup U_{w_N}$ 
contains $C \cap S^{n-1}$.
We claim that 
\begin{equation} \labell{finite intersection}
 C = \bigcap\limits_{\substack{w \in W \\ 1 \leq j \leq N_w}} 
     H_w^j .
\end{equation}

Fix $w \in W$ and $1 \leq j \leq N_w$.
Let $c$ be a point in $C$.
Let $c'$ be a point in $U_w \cap C \cap \del H_w^j$.
Because $C$ is convex, the segment $[c,c']$
is contained in $C$.  Because $U_w$ is open,
interior points of the segment that are sufficiently 
close to $c'$ are in $U_w$. Let $c''$ be such a point.
Because $c''$ is in $U_w \cap C$, it is in $H_w^j$.
Finally, because $c'$ is on the boundary of the half--space
and $c''$ is in the half--space, $c$ is also in the half--space.
Thus, $C \subset H_w^j$.

Denote the right hand side of~\eqref{finite intersection}
by $C_\text{RHS}$. We have shown that $C \subset C_\text{RHS}$.
Because $C$ is closed in $\R^n$, it is closed in $C_\text{RHS}$.
Because $C$ is the union of the sets $U_w \cap C$ for $w \in W$, 
and because $U_w \cap C$ is open in $C_w$, hence in $C_\text{RHS}$,
we deduce that $C$ is open in $C_\text{RHS}$.
Because $C_\text{RHS}$ is convex, hence connected,
and $C$ is a nonempty subset that is both closed and open,
$C$ is equal to $C_\text{RHS}$.
\end{proof}

The ``convexity package" for contact momentum maps
is given in parts (2), (3), (4), and (6) of the following theorem.

\begin{Theorem} \labell{contact main} 
Let a torus $T$ act on a cooriented compact connected contact manifold $M$ 
with contact momentum map $\Psi \co M\times \Rplus \to \ft^*$.
Assume that the action is effective and the torus has dimension
greater than 2.  Then
\begin{enumerate}
\item
Let $y_0$ and $y_1$ be any two points in $M \times \Rplus$.
\begin{itemize}
\item[\dash]
If the action is transverse ($0 \not\in \image\Psi$),
assume that 
the origin is not contained in the segment $[\Psi(y_0), \Psi(y_1)]$.
\item[\dash]
If the action is not transverse ($0 \in \image\Psi$),
assume that $\Psi(y_0)$ and $\Psi(y_1)$ are not both zero.
\end{itemize}
Then there exists a path $\gamma \co [0,1] \to M \times \Rplus$
such that $\gamma(0) = y_0$ and $\gamma(1)=y_1$
and such that $\Psi \circ \gamma \co [0,1] \to \ft^*$
is a weakly monotone parametrization of the
(possibly degenerate) segment $[\Psi(y_0),\Psi(y_1)]$.
\item The momentum map $\Psi$ is open as a map to its image. 	
\end{enumerate}
Consequently,
\begin{enumerate}
\setcounter{enumi}{2}	
\item The momentum cone $C(\Psi)$ is convex.
\item The nonzero level sets, $\Psi\Inv(\mu)$, for $\mu\neq 0$, 
are connected.
\item Let $A$ be a convex subset of $\ft^*$. 
\begin{itemize}
\item[\dash]
If the action is transverse, suppose that $0 \not\in A$.
\item[\dash]
If the action is not transverse, suppose that $A \neq \{ 0 \}$.
\end{itemize}
Then the preimage $\Psi\Inv(A)$ is connected. 
\end{enumerate}
Moreover,
\begin{enumerate}
\setcounter{enumi}{5}	
\item
The momentum cone $C(\Psi)$ is a convex polyhedral cone.
\end{enumerate}
\end{Theorem}

\begin{proof}
Parts (3), (4), and (5) of the theorem follow easily from Part (1).
We proceed to prove Parts (1), (2), and (6). 

Write the contact momentum map $\Psi\co M \times \Rplus \to \ft^*$ 
as $\Psi(x, t)=t\Psi_\alpha(x)$ where $\alpha$ is an invariant 
contact one--form 
and $\Psi_\alpha^X=\alpha(X_M)$ is the $\alpha$--momentum map. 
Choose a metric on $\ft^*$, denote the unit sphere by $S(\ft^*)$,
and let $\Psibar := \Psi / \| \Psi \| \co 
\left( M \ssminus \Psi_\alpha\Inv(0) \right) \times \Rplus \to S(\ft^*)$.

Because the $T$ action is effective on $M$, it is effective
on $M \times \Rplus$.  Because $M$ is connected, so is $M\times \Rplus$.
Because $T$ is compact and abelian, by the principal orbit type theorem, 
there exists an invariant open dense subset of $M\times \Rplus$ 
on which the action is free. Wherever the action is free, the momentum map 
$\Psi$ is a submersion. Hence, the $\Psi$--image of any open subset 
of $M \times \Rplus$ is not contained in a proper subspace of $\ft^*$.
Because $\dim T >2$, and because $\Psi(x,t) = t \Psi_\alpha(x)$,
this implies that the $\Psi_\alpha$--image of any open subset of $M$
is not contained in a two dimensional subspace of $\ft^*$.

\medskip

Suppose that the action is transverse ($ 0 \not\in \Psi_\alpha(M))$.

Let $(x,\lambda)$ be a point in $M \times \Rplus$. 
By Lemma~\ref{Psi local}, there exists a neighborhood $U$ such 
that the map $\Psi|_U$ is convex, $\Psi(U)$ is a relatively open subset 
in a convex polyhedral cone $C_{(x, \lambda)}$, and the map 
$\Psibar|_U$ is open as a map to its image. 
Parts (1) and (2) of the theorem follow from Proposition~\ref{stretch},
applied to the space $M$ and the map $\Psi_\alpha$.

Because $M$ is compact and $0 \not\in\image \Psi_\alpha$, 
the momentum map $\Psi$ is proper as a map to $\ft^* \ssminus \{ 0 \}$,
so its image is closed in $\ft^* \ssminus \{ 0 \}$.
Part (6) of the theorem then follows from Lemma~\ref{convex polyhedral},
applied to the momentum cone $C(\Psi) = \{ 0 \} \cup \image \Psi$.

\medskip

Now suppose that the action is not transverse 
$(0\in \text{image}\Psi_\alpha)$.

\begin{itemize}
\item
Let $x$ be a point of $M$ with $\Psi_\alpha(x) = 0$.
By Lemma~\ref{Psialpha local},
there exists a neighborhood $U_x$ of $x$ in $M$
and a convex polyhedral 
cone $C_x$ in $\ft^*$ with vertex at the origin such that
\begin{itemize}
\item[\dash]
the image $\Psi_\alpha(U_x)$ is an open subset of $C_x$;
\item[\dash]
the map $\Psi_\alpha|_{U_x} \co U_x \to \Psi_\alpha(U_x)$
is open, has the weak path lifting property,
and its level sets are path connected.
\end{itemize}
Because $\Psi_\alpha(U_x)$ is not contained
in a two dimensional subspace of $\ft^*$
but is contained in $C_x$,
\begin{itemize}
\item[\dash]
the cone $C_x$ is not contained in a two dimensional subspace of $\ft^*$.
\end{itemize}

\item
Let $(x,\lambda)$ be a point of $M \times \Rplus$ with $\Psi_\alpha(x) \neq 0$.
By Lemma~\ref{Psi local}, every neighborhood of $(x, \lambda)$ in 
$M \times \Rplus$ contains a smaller neighborhood $U$ such that the map 
$\Psi|_U$ 
is convex and the map $\Psibar|_U$ is (defined and) open as a map to its image. 
\end{itemize}
Parts (1), (2), and (6) of the theorem then follow from 
Proposition~\ref{with zero'}.
\end{proof}

\section{Examples}
\labell{sec:examples}

In the examples further below, we will need to know
that a contact manifold is determined by its symplectization
together with the $\Rplus$ action on the symplectization.
We will also need to use the ``contact cutting" construction.
These are summarized in the following remark.

\begin{Remark} \labell{symplectic cone}
Let $M$ be a $2n+1$ dimensional manifold,
$\pi \co Q \to M$ a principal $\Rplus$ bundle,
and $\omega$ a symplectic form on the total space $Q$
that is homogeneous of degree one 
with respect to the principal $\Rplus$ action.
This structure is called a \emph{symplectic cone} \cite{GS:homogeneous}.

Let $\talpha = \iota_v \omega$, where $v$ is the vector field 
that generates the principal $\Rplus$ action. 
Then there exists a unique contact distribution $\xi$ on $M$
and a unique diffeomorphism from the symplectization $\xi^0_+$ onto $Q$
that respects the projection maps to $M$ and such that
the pullback of $\talpha$ is the tautological one--form
on the subset $\xi^0_+$ of $T^*M$.

A torus $T$ action on $Q$ that commutes with the $\Rplus$ action
and preserves $\omega$ descends to an action on the contact manifold
$(M,\xi)$.  The pullback to $Q$ of the contact momentum map is the map
$\Psi \co Q \to \ft^*$
given by $\Psi^X = \iota_{X_Q} \talpha$,
where $X_Q$, for $X \in \ft$,
are the vector fields on $Q$ that generate the action.

Let $i \co S^1 \hookrightarrow T$ be a subcircle 
and $i^* \co \ft^* \to \R$ the projection on the dual of the
Lie algebra.  Performing on $Q$ the symplectic cutting construction
with respect to this circle action
yields the symplectization  of the \emph{contact cut} 
of $M$; cf.\ \cite[Theorem~6]{geiges}
and \cite[Theorem~2.10]{lerman:contact-cuts}.
Its momentum map image is the intersection of $\image \Psi$
with the closed half--space $\{ i^* \geq 0 \} $ of~$\ft^*$.
\end{Remark}

\begin{Example} \labell{ex:sphere}
Consider $\R^{2n} \ssminus \{ 0 \}$ as a principal $\Rplus$ bundle
over $S^{2n-1}$, where $t \in \Rplus$ acts by $x \mapsto \sqrt{t} x$ 
and where the map to $S^{2n-1}$ is $x \mapsto x/\| x \|$,
and with the standard symplectic structure.
This is the symplectization of the standard contact structure on $S^{2n-1}$
(cf.\ Remark~\ref{symplectic cone}).
The standard (Hopf) circle action has momentum map
$x \mapsto \| x \|^2/2 $ with image $(0,\infty)$.
The opposite circle action has momentum map with image $(-\infty,0)$.
\end{Example}

\begin{Example} \labell{disconnected}
Consider the torus $T^k = (S^1)^k$; identify its cotangent bundle
with $T^k \times \R^k$.
Let $Q$ be the complement of the zero section:
$Q = T^k \times (\R^k \ssminus \{ 0 \})$.
The subtorus $T^{k-1} \times \{ 1 \}$ acts on $Q$ with momentum map
$\Psi \co (a,x) \mapsto \pi(x)$
where $\pi \co \R^k \to \R^{k-1}$ is the projection to the first
$(k-1)$ coordinates.
The level sets of this momentum map are
$$ \Psi\Inv(\beta) = \begin{cases}
 T^k \times \{ \beta \} \times \R &
                        \text{ if } \beta \neq 0 \text{ in } \R^{k-1} \\
 T^k \times \{ 0 \} \times (\R \ssminus \{ 0 \}) &
                        \text{ if } \beta = 0 \text{ in } \R^{k-1}.
\end{cases} $$
In particular, the zero level set is not connected.
This is an example of a contact momentum map: $Q$ is the symplectization
of the unit sphere bundle in the cotangent bundle;
cf.\ Remark~\ref{symplectic cone}.
\end{Example}

\begin{Example} \labell{disconnected2}
Begin with the symplectic manifold $Q$ of Example~\ref{disconnected}.
By performing the symplectic cutting construction
with respect to the circle $\{ 1 \}^{k-2} \times S^1 \times \{ 1 \}$,
we obtain a new symplectic manifold, $Q_\cut$,
still with a $(k-1)$ dimensional torus action and a
momentum map $\Psi_\cut \co Q_\cut \to \R^{k-1}$,
but its momentum map image is now the closed upper half--space in $\R^{k-1}$
and not all of $\R^{k-1}$.
This is the contact momentum map for the contact manifold obtained
from the unit sphere bundle in the cotangent bundle
by ``contact cutting"; cf.\ Remark~\ref{symplectic cone}.
The zero level set of the momentum map is still disconnected.
(Contrast with Remark~\ref{symplectic proper}.)
\end{Example}

\begin{Example}[circle actions]\labell{ex:dim one}
Suppose that $\dim T = 1$ and the contact manifold $M$ is connected.
If $T$ acts effectively,  
the momentum map image must be one of the following sets:
$(-\infty,0)$, $(-\infty,0]$, $(0,\infty)$, $[0,\infty)$,
or all of $\R$. 
As seen in Example~\ref{ex:sphere}, and in Examples~\ref{disconnected} 
and~\ref{disconnected2} with $k=2$, all these sets occur as images,
and the zero level set need not be connected.
\end{Example}

\begin{Example}[$\dim T=2$] \labell{ex:dim two}
For a two dimensional torus, the momentum cone need not be convex,
and the level sets of the momentum map need not be connected.
To see this, we
begin with the \emph{noncompact} manifold $\R \times (S^1)^2$,
with $(S^1)^2$ acting by rotations of the second component,
and with the contact one--form 
$\alpha = \cos t d\theta_1+\sin t d\theta_2$
and the $\alpha$--momentum map
$(t, e^{i\theta_1}, e^{i\theta_2}) \mapsto (\cos t, \sin t)$.
For every positive integer $n$, this descends to a contact one--form
and torus action
on the compact manifold $M_n = \R/(2\pi n \Z) \times (S^1)^2$.
The image of the contact momentum map $M_n \times \Rplus \to \R^2$
is $\R^2 \ssminus \{0\}$,  
and every nonempty level set has $n$ connected components.
Alternatively, for every interval $[a,b] \subset \R$
such that $\tan(a)$ and $\tan(b)$ are rational,
contact cutting in neighborhoods of $\{ a \} \times (S^1)^2$ 
and $\{ b \} \times (S^1)^2$ 
produces a contact one--form and torus action
on the lens space $M_{[a,b]}$,
obtained from the manifold with boundary $[a,b] \times (S^1)^2$
by collapsing circles in the two components of the boundary
by two circle subgroups. 
The image of the contact momentum map 
is $\{ (r\cos t,r\sin t) \ | \ a \leq t \leq b \text{ and } r > 0 \}$.
If $\pi < b-a < 2\pi$, the momentum cone is not convex. 
If $b-a \geq 2\pi$, 
the contact momentum map is not open as a map to its image. 
These examples are due to Eugene Lerman~\cite{lerman:contact-cuts}. 
Also see Example~\ref{dim2}.  
We summarize this in Table \ref{table 1}. 
\begin{table}[htb]
\begin{center}
\begin{tabular}{|l|c|c|c|}
\hline
 & convexity (C1) & connectedness (C2) & openness (C3) \\
\hline
$M_n$ \ , \ $n=1$ & $\checkmark$ & $\checkmark$ & $\checkmark$ \\ 
$M_n$ \ , \ $n\geq 2$ & $\checkmark$ & $\times$ & $\checkmark$ \\ 
\hline
$M_{[a,b]}$ \ , \ $0 < b-a \leq \pi$ & 
  $\checkmark$ & $\checkmark$ & $\checkmark$ \\
$M_{[a,b]}$ \ , \ $\pi < b-a < 2\pi$ & 
  $\times $ & $\checkmark$ & $\checkmark$ \\
$M_{[a,b]}$ \ , \ $ b-a \geq 2\pi$ & 
  $\checkmark$ & $\times$ & $\times$ \\
\hline
\end{tabular}
\end{center}
\caption{}\labell{table 1}
\end{table}
\end{Example}

We currently do not know whether the contact momentum map
for a \emph{nontransverse} $T$ action 
on a compact connected contact manifold 
can have disconnected nonzero level sets when $\dim T = 2$.  
But it \emph{can} have disconnected nonzero level sets when $\dim T = 1$:

\begin{Example}[nonzero level sets and openness for circle actions] 
\labell{level}
For the manifolds of Example \ref{ex:dim two},
restrict the torus action to an action of the circle 
$S^1 \times \{ 1 \}$; the momentum map gets composed
with the projection map $(x,y) \mapsto x$.
For the manifold $M_n$,
the contact momentum map is open and its image is $\R$;
the number of connected components of $\Psi\Inv(x)$
is $n$ if $x \neq 0$ and $2n$ if $x=0$. 
For the manifold $M_{[a,b]}$,
where $b = b' + 2\pi k$ with $k$ a nonnegative integer,
the numbers of connected components of $\Psi\Inv(x)$ 
and the images of $\Psi$ when $k=0$
are given in Table \ref{table 2}.
\begin{table}[htb]
\begin{center}
\begin{tabular}{|rcr|c|c|c|c|}
\hline
  &&& $x<0$ & $x=0$ & $x>0$ & $\image\Psi$\\
  &&& & & & when $k=0$ \\
\hline
$a=-\pi/2$ &,& $-\pi/2<b'<\pi/2$            & k & 1+2k & 1+k & $[0,\infty)$ \\
$a=-\pi/2$ &,& $\phantom{-\pi/2<}b'=\pi/2$  & k & 2+2k & 1+k & $[0,\infty)$ \\
$a=-\pi/2$ &,& $\phantom{-}\pi/2<b'<3\pi/2$ & 1+k & 2+2k & 1+k & $\R$ \\
$a=-\pi/2$ &,& $\phantom{-\pi/2<}b'=3\pi/2$ & 1+k & 3+2k & 1+k & $\R$ \\
$-\pi/2<a<\pi/2$ &,& $a<b'<\pi/2 $          & k & 2k & 1+k & $(0,\infty)$ \\
$-\pi/2<a<\pi/2$ &,& $\phantom{a<}b'=\pi/2$ & k & 1+2k & 1+k & $[0,\infty)$ \\
$-\pi/2<a<\pi/2$ &,& $\pi/2<b'<3\pi/2$      & 1+k & 1+2k & 1+k & $\R$ \\
$-\pi/2<a<\pi/2$ &,& $\phantom{\pi/2<}b'=3\pi/2$& 1+k & 2+2k & 1+k & $\R$ \\ 
$-\pi/2<a<\pi/2$ &,& $3\pi/2<b'\leq a+2\pi$     & 1+k & 2+2k & 2+k & $\R$ \\
\hline
\end{tabular}
\end{center}
	\caption{}\labell{table 2}
\end{table}
If $k \geq 1$, then $\image \Psi = \R$.
When we replace $[a,b]$ by $[a+\pi n,b+\pi n]$,
if $n$ is even, we get the same values, and if $n$ is odd,
the number of connected components for $x<0$ is switched
with the number for $x>0$ and $\image \Psi$ transforms
by $x \mapsto -x$.
When the image of the momentum map is $\R$
and $a$ or $b$ is equal to $\pi/2$ modulo $\pi \Z$, 
the contact momentum map $\Psi$ is not open as a map to its image;
in all other cases, the contact momentum map is open as a map 
to its image.
We summarize this in Tables \ref{table 3} and \ref{table 4}. 
(The convexity (C1) is automatic.) 
\begin{table}[htb]
\begin{center}
\begin{tabular}{|c|c|c|c|}
\hline
$M_n$ & convexity (C1) & connectedness of nonzero level sets (C2) & 
openness (C3) \\
\hline
 $n=1$     & \checkmark & \checkmark & \checkmark \\ 
 $n\geq 2$ & \checkmark & $\times$ & \checkmark \\ 
\hline
\end{tabular} 
\end{center}
	\caption{}\labell{table 3}
\end{table}

\begin{table}[htb]
\begin{center}
\begin{tabular}{|rrc|c|c|c|}
\hline
\multicolumn{3}{|c|}{$M_{[a,b]}$\ ,\ $b=b'+2\pi k$} & (C1) & (C2) & (C3) \\
\hline
$a=-\pi/2$ \ ,& $-\pi/2<b'\leq\pi/2$\ ,& \hfill $k=0$ & 
 \checkmark & \checkmark & \checkmark \\
  && $k \geq 1$ & \checkmark & $\times$ & $\times$ \\
\hline
$a=-\pi/2$\ ,& $\pi/2< b' \leq 3\pi/2$\ ,& $k = 0$ &
 \checkmark & \checkmark & $\times$ \\
 && $k \geq 1$ & \checkmark & $\times$ & $\times$ \\
\hline
$-\pi/2<a<\pi/2$\ ,& $a<b'<\pi/2 $\ ,&  $k=0$ &
 \checkmark & \checkmark & \checkmark \\
  && $k \geq 1$ & \checkmark & $\times$ & \checkmark \\
\hline
$-\pi/2<a<\pi/2$\ ,& $\phantom{a<}b'=\pi/2$\ ,& $k=0$ &
 \checkmark & \checkmark & \checkmark \\
  && $k \geq 1$ & \checkmark & $\times$ & $\times$ \\
\hline
$-\pi/2<a<\pi/2$\ ,& $\pi/2<b'<3\pi/2$\ ,& $k=0$ &
 \checkmark & \checkmark & \checkmark \\
  && $k \geq 1$ & \checkmark & $\times$ & \checkmark \\
\hline
$-\pi/2<a<\pi/2$\ ,& $\phantom{\pi/2<}b'=3\pi/2$\ ,& $k = 0$ &
 \checkmark & \checkmark & $\times$ \\
 && $k\geq 1$ & \checkmark & $\times$ & $\times$ \\
\hline
$-\pi/2<a<\pi/2$\ ,& $3\pi/2<b'\leq a+2\pi$\ ,& $k \geq 0$ &
 \checkmark & $\times$ & \checkmark  \\
\hline
\end{tabular}
\end{center}
	\caption{}\labell{table 4}
\end{table}
\end{Example}

\end{document}